\documentclass[10pt,oneside,reqno]{amsart}

\usepackage[utf8]{inputenc}
\usepackage[T1]{fontenc}
\usepackage[russian,ngerman,english]{babel}
\usepackage{lmodern}

\usepackage{hyperref}

\usepackage{amsmath}
\usepackage{amsthm}
\usepackage{amsfonts}
\usepackage{amssymb}
\usepackage{amscd}
\usepackage{amsbsy}

\usepackage{pdfpages}
\usepackage{fancyhdr}
\usepackage{geometry}
\usepackage{setspace}
\usepackage{xcolor}
\usepackage{pifont}

\usepackage{graphicx}
\usepackage{tabularx}
\usepackage{epsfig}
\usepackage[all]{xy}
\usepackage{tikz}
\usetikzlibrary{matrix}

\usepackage{fixmath}

\usepackage{appendix}

\usepackage{enumerate}

\usepackage[sort, numbers]{natbib}
\usepackage{bibentry}

\newtheorem{theorem}{Theorem}[section]

\newtheorem{lemma}[theorem]{Lemma}
\newtheorem{proposition}[theorem]{Proposition}

\newtheorem{corollary}[theorem]{Corollary}

\theoremstyle{definition}
\newtheorem{definition}[theorem]{Definition}

\newtheorem*{remark}{Remark}

\newtheorem{example}[theorem]{Example}

\numberwithin{equation}{section}

\newcommand{\Z}{{\mathbb Z}}

\newcommand{\R}{{\mathbb R}}
\renewcommand{\C}{{\mathbb C}}

\newcommand{\E}{{\mathsf E}}

\newcommand{\tr}{\operatorname{tr}}

\newcommand{\cD}{{\mathcal{D}}}

\newcommand{\cF}{{\mathcal{F}}}

\newcommand{\cP}{{\mathcal{P}}}

\newcommand{\e}{{\varepsilon}}

\newcommand{\z}{z}
\newcommand{\g}{\gamma}

\renewcommand{\l}{\lambda}

\renewcommand{\b}{\beta}
\renewcommand{\d}{\delta}
\renewcommand{\t}{{\theta}}

\newcommand{\height}{s}

\newcommand{\dd}{\mathrm{d}}

\renewcommand{\Re}{\operatorname{Re}}
\renewcommand{\Im}{\operatorname{Im}}
\renewcommand{\Cap}{\operatorname{Cap}}

\newcommand{\wlim}{\operatorname*{w-lim}}

\renewcommand{\i}{\infty}

\newcommand{\ess}{\text{\rm{ess}}}

\newcommand{\bbN}{\mathbb{N}}
\newcommand{\bbZ}{\mathbb{Z}}

\newcommand{\bbR}{\mathbb{R}}
\newcommand{\bbC}{\mathbb{C}}
\newcommand{\bbE}{\mathbb{E}}

\newcommand{\loc}{\mathrm{loc}}

\allowdisplaybreaks

\title{Stahl--Totik regularity for continuum Schr\"odinger operators}

\thanks{B.E.\ was supported by Austrian Science Fund FWF, project no: J 4138-N32.}
\thanks{M.L.\ was supported in part by NSF grant DMS--1700179.}

\author{Benjamin~Eichinger and Milivoje~Luki\'c}
\begin{document}

\address{Department of Mathematics, Rice University MS-136, Box 1892,
Houston, TX 77251-1892, USA and Institute of Analysis, Johannes Kepler University of Linz, 4040 Linz, Austria.}
\email{benjamin.eichinger@rice.edu}
\address{Department of Mathematics, Rice University MS-136, Box 1892,
Houston, TX 77251-1892, USA.}
\email{milivoje.lukic@rice.edu}

\begin{abstract}
We develop a theory of regularity for continuum Schr\"odinger operators based on the Martin compactification of the complement of the essential spectrum. This theory is inspired by Stahl--Totik regularity for orthogonal polynomials, but requires a different approach, since Stahl--Totik regularity is formulated in terms of the potential theoretic Green function with a pole at $\infty$, logarithmic capacity, and the equilibrium measure for the support of the measure, notions which do not extend to the case of unbounded spectra. For any half-line Schr\"odinger operator with a bounded potential (in a locally $L^1$ sense), we prove that its essential spectrum obeys the Akhiezer--Levin condition, and moreover, that the Martin function at $\infty$ obeys the two-term asymptotic expansion $\sqrt{-z} + \frac{a}{2\sqrt{-z}} + o(\frac 1{\sqrt{-z}})$ as $z \to -\infty$. The constant $a$ in that expansion plays the role of a renormalized Robin constant suited for Schr\"odinger operators and enters a universal inequality $a \le \liminf_{x\to\infty} \frac 1x \int_0^x V(t)\dd t$. This leads to a notion of regularity, with connections to the root asymptotics of Dirichlet solutions and zero counting measures. We also present applications to decaying and ergodic potentials.
\end{abstract}

\maketitle

\section{Introduction} 
The goal of this paper is to develop a theory of Stahl--Totik regularity suitable for continuum Schr\"odinger operators; it is natural for this topic to work in the half-line setting, so our Schr\"odinger operators are unbounded self-adjoint operators on $L^2((0,\infty))$, corresponding formally to
\[
L_V = - \frac{d^2}{dx^2}+V.
\]
The potential $V$ will always be real-valued and assumed to be uniformly locally integrable, i.e.
\begin{equation}\label{L1locunif}
\sup_{x \ge 0} \int_x^{x+1} \lvert V(t) \rvert\dd t < \infty
\end{equation}
(in particular, $0$ is a regular endpoint and $+\infty$ is a limit point endpoint in the sense of Weyl). We set the Dirichlet boundary condition at $0$,  so the domain of the operator is 
\[
D(L_V) = \{ f \in L^2((0,\infty)) \mid f \in W^{2,1}_\loc ([0,\infty)), -f'' + V f \in L^2((0,\infty)), f(0) = 0\}
\]
where $W^{2,1}_\loc([0,\infty))$ denotes the set of functions such that $f \in W^{2,1}([0,x])$ for all $x < \infty$, i.e., $f'' \in L^1([0,x])$ for all $x < \infty$.

The connection of orthogonal polynomials and potential theory goes back at least to the work of Faber and Szeg\H o \cite{Fab20,Sz24}. For further references on the subject we refer to the paper of Simon \cite{Simon07} and the monograph of Stahl and Totik \cite{StahlTotik92}. Building on important work of Ullman \cite{Ullman72}, Stahl and Totik developed a comprehensive theory for orthogonal polynomials for arbitrary measures with compact support in $\C$. It is shown that the asymptotic behavior of the orthogonal polynomials is intimately related with so-called  Stahl--Totik regularity of the measure. 
For orthogonal polynomials on the real line, the theory provides a universal inequality between the Jacobi coefficients of a compactly supported measure and the logarithmic capacity of its topological support $\E$, and the measure is defined to be Stahl--Totik regular if equality holds. This theory is built on potential theoretic notions, such as Green functions on the domain $\Omega = \hat\bbC \setminus \E$ with the pole at $\infty$, logarithmic capacity, and equilibrium measures 
 --  objects which are undefined for unbounded sets $\E$, and therefore not applicable to continuum Schr\"odinger operators.

In this paper, we develop the corresponding theory for Schr\"odinger operators. Martin functions \cite{Martin41,ArmGar01} serve as the counterpart of Green functions, corresponding to boundary points $z_0 \in \partial\Omega$ instead of internal points $z_0 \in \Omega$; but whereas the Green function is defined with an explicit logarithmic singularity at $z_0$, the existence and behavior of Martin functions is more varied. If $\E \subset \bbR$ is a closed unbounded set, $\infty$ is a boundary point of the Denjoy domain $\Omega = \bbC \setminus \E$. If this domain is Greenian, associated to the boundary point $\infty$ is a cone of dimension $1$ or $2$ of positive harmonic functions in $\Omega$ which are bounded on bounded sets and vanish at every Dirichlet-regular point of $\E$. The cone is spanned by the minimal Martin functions at $\infty$ \cite{AkhieLevin,Anc79,Be80,GardSjoed09}. In particular, if $\min \E > -\infty$, the cone is of dimension $1$, and the Martin function at $\infty$ is determined uniquely up to normalization; we denote it by $M_\E$ and simply call it the Martin function from now on.

For sets with $\min \E > -\infty$, the Akhiezer--Levin condition is
\begin{equation}\label{AkhiezerLevin1}
\lim_{z\to -\infty} \frac{M_\E(z)}{\sqrt{-z}} > 0
\end{equation}
(by general principles, the limit exists with a value in $[0,\infty)$). For Akhiezer--Levin sets, we will normalize the Martin function so that the limit in \eqref{AkhiezerLevin1} is equal to $1$.

For a potential bounded in the sense \eqref{L1locunif}, the spectrum $\sigma(L_V)$ is a closed subset of $\bbR$ bounded below but not above, so the above definitions are applicable. It will be noted that isolated points of the set don't affect the Martin function, so we can equally well use $\E = \sigma_\ess(L_V)$ in what follows (more generally, $M_{\E_1} = M_{\E_2}$ if the symmetric difference of $\E_1$ and $\E_2$ is a polar set).

In spectral theory, Martin functions first appear implicitly, in the classical work of Marchenko--Ostrovski \cite{MarOst75} classifying the spectra of periodic Schr\"odinger operators. In this work, the discriminant of a $1$-periodic operator is expressed in the form $\Delta(z) = 2 \cos(\Theta(z))$, and it can be recognized that $\Im \Theta(z)$ is the Martin function at $\infty$ for the periodic spectrum. The explicit use of Martin functions in spectral theory starts with works of Yuditskii and coauthors \cite{SodinYuditskii95,DamanikYuditskii16,EVY19}, through inverse spectral theoretic studies associated to Dirichlet-regular spectra obeying a Widom condition and finite gap length conditions.

In contrast to the previous works, our first theorem is a set of universal properties of the spectra of Schr\"odinger operators obeying \eqref{L1locunif}; note that a boundedness condition such as \eqref{L1locunif} is essential for the following theory, since potentials going to $-\infty$ or $+\infty$ can give spectrum equal to $\bbR$ or spectrum which is a polar set.

\begin{theorem}\label{thm:MainAkhiezerLevin}
For any potential $V$ obeying \eqref{L1locunif} and $\E = \sigma_\ess(L_V)$, the domain $\Omega = \bbC \setminus \E$ is Greenian, $\infty$ is a Dirichlet-regular point for $\Omega$, $\Omega$ obeys the Akhiezer--Levin condition, and there exists $a_\E \in \bbR$ such that the Martin function has the asymptotic behavior
\begin{align}\label{Martinexp2term}
M_\E(z)=\Re\left(\sqrt{-z}+\frac{a_\E}{2\sqrt{-z}}\right)+o\left(\frac{1}{\sqrt{|z|}}\right),
\end{align}
as $z \to \infty$, $\arg z \in [\delta,2\pi - \delta]$, for any $\delta > 0$.
\end{theorem}

Each of the conclusions of this theorem is strictly stronger than the previous; we will point out examples in Section 2. In particular, the second term of the expansion \eqref{Martinexp2term} is not an automatic property of Akhiezer--Levin sets, but rather an added feature corresponding to spectra of Schr\"odinger operators. It should be emphasized that spectra of Schr\"odinger operators with bounded potentials can be very thin in the sense that they can even have zero Hausdorff dimension \cite{DFL17} and zero lower box counting dimension \cite{DFG19}, while our result is a universal ``thickness'' result in the perspective of Martin function.

In the references given above, the Martin function was used in spectral theory as a positive harmonic function in $\Omega$ that vanishes on the boundary. In fact, Martin theory provides a whole kernel $M(z,x)$ on $\Omega\times(\hat\Omega\setminus \{z_*\})$, where $\hat\Omega$ denotes the Martin compactification of $\Omega$ and $z_*\in\Omega$ is a normalization point. If $\partial^M_1\Omega$ denotes the so-called minimal Martin boundary of $\Omega$, then for every positive harmonic function $h$ on $\Omega$ there exisists a unique finite measure $\nu$ such that
\begin{align*}
h(z)=\int_{\partial^M_1\Omega}M(z,x)\dd \nu(x).
\end{align*}
We will provide more details and precise definitions in Section \ref{sec:Martinfunction}. It is new to combine this theory with the spectral theory of unbounded self-adjoint operators and it was crucial for the proof of Theorem \ref{thm:MainAkhiezerLevin}.

Under strong a priori assumptions on the spectrum, expansions of the form \eqref{Martinexp2term} have previously appeared in the spectral theory literature \cite{MarOst75}. Namely, the set $\E$ is closed so it can be written in the form
\begin{align}\label{eq:spectrumGaps}
\E=[b_0,\infty)\setminus\bigcup_{j=1}^N (a_j,b_j)
\end{align} 
where $j$ indexes the ``gaps'', i.e., connected components of $[b_0,\infty) \setminus \E$, and $N$ is finite or $\infty$. If $\sum_{j} (b_j - a_j) < \infty$, the Martin function has an expansion \eqref{Martinexp2term} with $a_\E = b_0 + \sum_{j} (a_j + b_j - 2 c_j)$ (Lemma~\ref{lemma:constantfromfinitegaplength}) by harmonic/complex theoretic arguments. In contrast, our Theorem~\ref{thm:MainAkhiezerLevin} is not a purely complex theoretic result; its proof is a combination of spectral theoretic arguments and the theory of the Martin boundary of Denjoy domains.

In particular, Theorem~\ref{thm:MainAkhiezerLevin} associates to the set $\E$ the real-valued constant $a_\E$, which will serve as a substitute for the Robin constant and enter the following universal inequality:

\begin{theorem}\label{thm11}
If $V$ is a potential obeying \eqref{L1locunif} and $\E = \sigma_\ess(L_V)$, then
\begin{equation}\label{aEinequality1}
a_\E \le \liminf_{x\to\infty} \frac 1x \int_0^x V(t) \dd t.
\end{equation}
\end{theorem}

For any $z\in \bbC$, the Dirichlet eigensolution is the solution of the initial value problem
\[
-\partial_x^2 u(x,z) + V(x) u(x,z) = z u(x,z), \qquad u(0,z)=0,\qquad (\partial_x u)(0,z) = 1.
\]
Our next result is that the Martin function provides a universal lower bound on the growth rate of the Dirichlet solution.

\begin{theorem}\label{thm13}
If $V$ is a potential obeying \eqref{L1locunif} and $\E = \sigma_\ess(L_V)$, then
\begin{align*}\label{eq:thm13}
M_\E(z) \le \liminf_{x\to\infty} \frac 1x \log \lvert u(x,z) \rvert, \qquad \forall z\in \bbC \setminus [ \min \E, \infty).
\end{align*}
\end{theorem}

Exclusion of $[\min \E, \infty)$ in Theorem~\ref{thm13} is necessary because for $z\in (\min \E, \infty)$,  by Sturm oscillation theory \cite{Simon05}, the Dirichlet solution has infinitely many zeros.

\begin{definition}
The potential $V$ is regular if
\begin{equation}\label{eqn:regulardefn}
a_\E = \lim_{x\to\infty}  \frac 1x \int_0^x V(t) \dd t.
\end{equation}
\end{definition}

Of course, due to \eqref{aEinequality1}, this is equivalent to requiring that $a_\E \ge \limsup_{x\to\infty} \frac 1x \int_0^x V(t) \dd t$.

In our next theorem, we will characterize regularity in terms of root asymptotics for the Dirichlet eigensolutions. We say that a property holds a.e.\ on $\E$ with respect to harmonic measure if it holds away from a set $A\subset \E$ such that $\omega_\E(A,z_0)=0$, where $\omega_\E(\cdot,z_0)$ denotes the harmonic measure of $\Omega$ evaluated at some $z_0 \in \Omega$. This condition is independent of the choice of $z_0\in \Omega$ since the harmonic measures are mutually absolutely continuous.

\begin{theorem}\label{thm14}
If $V$ is a potential obeying \eqref{L1locunif} and $\E = \sigma_\ess(L_V)$, the following are equivalent:
\begin{enumerate}[(i)]
\item $V$ is regular;
\item For every Dirichlet-regular $z\in \E$, $\limsup_{x\to\infty} \frac 1x \log \lvert u(x,z) \rvert \le 0$;
\item For a.e.\ $z\in \E$ with respect to harmonic measure, $\limsup_{x\to\infty} \frac 1x \log \lvert u(x,z) \rvert \le 0$;
\item For all $z\in \bbC_+$, $\limsup_{x\to\infty} \frac 1x \log \lvert u(x,z) \rvert \le M_\E(z)$;
\item For all $z\in \C$, $\limsup_{x\to\infty} \frac 1x \log \lvert u(x,z) \rvert \le M_\E(z)$;
\item $\lim_{x\to\infty} \frac 1x \log \lvert u(x,z) \rvert = M_\E(z)$ uniformly on compact subsets of $\bbC \setminus [\min \E, \infty)$.
\end{enumerate}
\end{theorem}
Since (v) or (vi) trivially imply (iv), (iv) is of interest as a criterion for establishing regularity of $V$, whereas (v), (vi) are of interest as consequences of regularity. Similarly, (ii) implies (iii), so (ii) is of interest as a consequence of regularity and (iii) as a condition for regularity. Instead of conditions (ii) and (iii), it would be customary to state the single condition that the inequality holds quasi-everywhere; this is between our conditions since the set of Dirichlet-irregular points is polar and polar sets have harmonic measure $0$. The benefit of (ii) is that it can be used pointwise (in particular, for a Dirichlet-regular set $\E$, the inequality holds everywhere on $\E$). More importantly, the benefit of (iii) is that the characterization in terms of harmonic measure will be essential for our proof of Theorem \ref{thm:Widom} below.

An analog of Stahl--Totik regularity for Schr\"odinger operators was previously conjectured by Simon \cite[Section 9]{Simon07}. It was conjectured that there should be a function $\Phi_\E$ associated to the essential spectrum $\E$, which would serve as a bound for the exponential growth rate of $u(x,z)$, and it was suggested that regularity for continuum Schr\"odinger operators can be defined by the condition $\limsup_{x\to\infty} \frac 1x \log \lvert u(x,z) \rvert = \Phi_\E(z)$. It was expected that this theory would require a condition on $\E$ such as finite gap length. Our results need no such assumptions, so they solve these conjectures in a far greater generality than they were even previously conjectured.
Among the properties expected in \cite{Simon07} for the function $\Phi_\E$ (now understood as the Martin function $M_\E$) is an asymptotic expansion $\Phi_\E(z) = \Re(\sqrt{-z})(1+o(1))$ near $-\infty$, and it is conjectured that this should improve to the asymptotic behavior $\Re \sqrt{-z} + o(1)$; this is motivated by the asymptotic behavior $\sqrt{-z} + o(1)$ of $m$-functions, proved by Atkinson \cite{Atkinson81}. While that asymptotic statement for individual $m$-functions cannot be improved for locally integrable potentials, we discover that due to averaging effects, the asymptotic behavior of our quantities improves to the form \eqref{Martinexp2term}. This discovery of \eqref{Martinexp2term} has enabled us to introduce the constant $a_\E$, which was not previously conjectured, and to use it for the robust general definition of regularity given above. Simon also conjectured the existence of an ``equilibrium measure'' for $\E$ which would be related to a deterministic density of states; we give the definition of this measure and solve the corresponding conjectures below.

The Martin function can be extended to a subharmonic function on $\C$, so it has a Riesz measure, given by 
\begin{align*}
\rho_\E = \frac{1}{2\pi}\Delta M_\E,
\end{align*}
which we will call the Martin measure of the set $\E$. Conversely, the Martin function has a Hadamard representation of the form
\[
M_\E(z) = M_\E(z_*) + \int_\E \log \left\lvert 1 - \frac{z-z_*}{t-z_*} \right\rvert \dd \rho_\E(t)
\]
where $z_* < \min \E$ is an arbitrary normalization point. The Martin measure will serve the same role in this theory that the logarithmic equilibrium measure serves for orthogonal polynomials. However, $\rho_\E$ is not defined with respect to any extremal property (and it is not even a finite measure), so different proofs will be needed in the current setting.

For any $x > 0$, let $\rho_x$ denote the zero counting measure for $u(x,z)$ divided by $x$,
\begin{equation}\label{eqn:rhoxdefn}
\rho_x = \frac 1x \sum_{z: u(x,z)=0} \delta_z.
\end{equation}
Note that $\rho_x$ is the Riesz measure of $\frac{1}{x}\log \lvert u(x,z)\rvert$. The limit of $\rho_x$ as $x \to \infty$, when it exists, is interpreted as a deterministic density of states associated to $V$. The Martin measure and the zero counting measures are related by the following pair of results:

\begin{theorem}\label{thm:dos1}
Assume $V$ is regular. Then $\rho_x$ converges to $\rho_\E$ as $x \to \infty$, in the weak-$*$ sense.
\end{theorem}

The following is a continuum analog of a result of Stahl--Totik \cite{StahlTotik92}:

\begin{theorem}\label{thm:dos2}
Assume that $V$ obeys \eqref{L1locunif} and let $\mu$ be a maximal spectral measure for $L_V$. Suppose that $\rho_x$ converges to $\rho_\E$ as $x \to \infty$ in the weak-$*$ sense. Then, either $V$ is regular, or there exists a polar Borel set $X$ such that $\mu(\bbR \setminus X) = 0$.
\end{theorem}

Of course, the statement  $\mu(\bbR \setminus X) = 0$ can be restated in the language of the Borel functional calculus as $\chi_{\bbR \setminus X}(L_V)=0$.

So far, we have seen that regularity of $V$ can be established from the root asymptotics of Dirichlet solutions. The next theorem shows that it can be established from spectral properties of the operator. It is the continuum counterpart of a theorem of Widom \cite{Widom67}. 

\begin{theorem}\label{thm:Widom}
	Let $\mu$ be a maximal spectral measure for $L_V$. If $\omega_{\E}(\cdot,z_0)$ for some $z_0 \in \C \setminus\E$ is absolutely continuous with respect to $\mu$, then $V$ is regular.
\end{theorem}

This theory leads to several new results even for the special case of half-line essential spectrum $[0,\infty)$; we present those as our first applications. If $V$ is a decaying potential in the sense
\begin{equation}\label{L1decaying}
\lim_{x\to\infty} \int_x^{x+1} \lvert V(t) \rvert \dd t  = 0
\end{equation}
then $\E = \sigma_\ess(L_V) = [0,\infty)$ by Blumenthal--Weyl \cite{Blum,Weyl}. It follows that $M_\E(z) = \Re \sqrt{-z}$. In particular, $a_\E = 0$, so immediately from the definition:

\begin{corollary}\label{cor:decayingregular}
If $V$ is a decaying potential in the sense \eqref{L1decaying}, then $V$ is regular with $\sigma_\ess(L_V) =[0,\infty)$.
\end{corollary}

Since harmonic measure for $\E = [0,\infty)$ is mutually absolutely continuous with $\chi_{(0,\infty)}(x) \dd x$, the following is an immediate consequence of Theorem~\ref{thm:Widom}:

\begin{corollary}\label{cor:Widomhalfline}
Assume that $V$ obeys \eqref{L1locunif} and denote by $\mu$ a maximal spectral measure for $L_V$. Denote by $\dd\mu = f \dd x + \dd \mu_s$ the Radon--Nikodym decomposition of $\mu$ with respect to Lebesgue measure. If $\sigma_\ess(L_V) = [0,\infty)$ and $f(x) > 0$ for Lebesgue-a.e.\ $x>0$, then $V$ is regular.
\end{corollary}

More generally, a version of Corollary \ref{cor:Widomhalfline} holds, whenever the harmonic measure for the domain $\C\setminus \E$ is absolutely continuous with respect to the Lebesgue measure  $\chi_\E(x)\dd x$. In particular, it holds for finite gap sets (i.e., when $N$ is finite in \eqref{eq:spectrumGaps}) and regular Parreau-Widom sets. If $\E$ is Dirichlet-regular, the Green function $G_\E(z,z_0)$, for $z_0<\min E$, has exactly one critical point $c_j\in(a_j,b_j)$ in each gap. If, in addition, the critical values of $G_\E(z,z_0)$ are summable, i.e., 
\begin{align*}
\sum_{j=1}^{\infty}G_\E(c_j,z_0)<\infty,
\end{align*}
we call $\E$ a regular Parreau-Widom set. In fact, the harmonic measure for the domain $\C\setminus \E$ is absolutely continuous with respect to the Lebesgue measure if and only if $\E$ satisfies a certain sector condition \cite[Theorem 4]{ErY12}. We will describe this generalization in Section \ref{sec:confMaps}.

Sparse potentials are not covered by Corollary~\ref{cor:decayingregular} or Corollary~\ref{cor:Widomhalfline}, but nonetheless provide additional examples of regular potentials:

\begin{example}\label{ex:Sparse1}
Let $W\in L^1((0,\infty))$ be compactly supported $W\geq 0$, $W\not\equiv 0$, $x_n \ge 0$ an increasing sequence such that $x_{n+1}-x_n\to\infty$ as $n\to\infty$ and $V(x)=\sum_nW(x-x_n)$. Then $V$ is regular with $\sigma_\ess(L_V) = [0,\infty)$.
\end{example}

The sparse potentials from Example~\ref{ex:Sparse1} are not decaying in the sense \eqref{L1decaying}, so Corollary~\ref{cor:decayingregular} does not have a converse; sparse potentials have purely singular spectrum by \cite{Pea78,LasSim99},  so Corollary~\ref{cor:Widomhalfline} does not have a converse.

However, we prove that Corollary~\ref{cor:decayingregular} has the following partial converse; we have already described Theorem~\ref{thm:MainAkhiezerLevin} as a universal thickness result about the spectrum, and the following result similarly guarantees presence of essential spectrum.

\begin{theorem}\label{corL1Cesaro} 
Assume that $V$ obeys \eqref{L1locunif} and that $\sigma_\ess(L_V) \subset [0,\infty)$. Then:
\begin{enumerate}[(a)]
\item $\liminf_{x\to \infty} \frac 1x \int_0^x V(t) \dd t \ge 0$; 
\item If $\liminf_{x\to \infty} \frac 1x \int_0^x V(t) \dd t \le 0$, then $\sigma_\ess(L_V) = [0,\infty)$;
\item If $\limsup_{x\to \infty} \frac 1x \int_0^x V(t) \dd t \le 0$, then $\sigma_\ess(L_V) = [0,\infty)$ and $V$ is regular.
\end{enumerate}
\end{theorem}

Part (a) can also be established by other means, but we include it for completeness. Parts (b) and (c) generalize known results giving sufficient conditions for $\sigma_\ess(L_V) = [0,\infty)$. In particular, Damanik--Remling \cite[Theorem 1.2]{DamanikRemlingDuke} showed that $\sigma_{\ess}(L_{\pm V})\subset [0,\infty)$ implies $\sigma_{\ess}(L_{V})=[0,\infty)$. Part (b) of our theorem is a strict generalization of that result; strict because it applies, e.g., to the sparse potentials of Example~\ref{ex:Sparse1} where \cite{DamanikRemlingDuke} does not (for a positive sparse potential $V$, $\min\sigma_\ess(L_{-V}) < 0$), and a generalization because $\sigma_{\ess}(L_{- V})\subset [0,\infty)$ implies $\limsup_{x\to \infty} \frac 1x \int_0^x V(t) \dd t \le 0$ (by (a) applied to $-V$), so our parts (b), (c) also apply to the potentials in \cite{DamanikRemlingDuke}. In particular, $\sigma_{\ess}(L_{\pm V})\subset [0,\infty)$ implies that $V$ is regular and $\sigma_{\ess}(L_{V})= [0,\infty)$.

In the theory of Jacobi matrices, a result of Simon \cite{Simon09} shows that a regular Jacobi matrix with essential spectrum $[-2,2]$ obeys a Ces\`aro--Nevai condition. The analog for Schr\"odinger operators is false -- the continuum setting allows rapid oscillations which can break any Ces\`aro-type decay in an $L^1$ sense:

\begin{example}\label{thmnotCesaro} 
The potential defined piecewise by $V(x) = (-1)^{\lfloor 2n(x-n) \rfloor}$ on $x \in [n-1,n)$ for integer $n$ is regular with $\sigma_\ess(L_V) = [0,\infty)$, but $\frac 1x \int_0^x \lvert V(t) \rvert \dd t \not \to 0$ as $x\to\infty$.
\end{example}

All objects considered above are deterministic (defined only in terms of a single half-line potential $V$), but for ergodic families of Schr\"odinger operators, they can be recognized almost surely as ergodic notions such as the Lyapunov exponent and the ergodic density of states, so our results can be interpreted in the ergodic setting. In the ergodic setting, it is natural to work with whole line potentials: let us consider a family $(V_\eta)_{\eta \in S}$ of real-valued potentials on $\bbR$ on a probability space $S$ which is metrically transitive with respect to a group of measure preserving transformations $\tau_y$ such that $V_{\tau_y \eta}(x) = V_\eta(x-y)$ and such that any measurable subset $A$ of $S$ which is invariant under all $\tau_y$ has probability $0$ or $1$. The group of transformations can be continuous (indexed by $y\in \bbR$) or discrete (indexed by $y\in \ell \bbZ$ for some $\ell > 0$); the former case includes almost periodic Schr\"odinger operators and the latter case includes many Anderson-type models studied in the literature \cite{Kir85, DamSimSto02}, including those with a periodic background. We also assume that $V_\eta$ almost surely obeys
\begin{equation}\label{eqn:ergodicbounded}
\sup_{x\in \bbR} \int_x^{x+1} \lvert V_\eta(t)\rvert\dd t < \infty;
\end{equation}
in fact, much of the literature on ergodic Schr\"odinger operators assumes bounded potentials. Let us denote by $H_{V_\eta}$ the self-adjoint operators on $L^2(\bbR)$ given by
\[
D(H_{V_\eta}) = \{ f \in L^2(\bbR) \mid f \in W^{2,1}_\loc(\bbR), -f''+V_\eta f \in L^2(\bbR) \}.
\]
In this ergodic setting, there is an almost sure spectrum,
\[
\E = \sigma(H_{V_\eta}) = \sigma_\ess(H_{V_\eta}), \qquad \text{for a.e. }\eta \in S,
\]
and the potentials $V_\eta$ have an almost sure Birkhoff average $\bbE(V)$,
\[
\bbE(V) = \lim_{x\to\infty} \frac 1x \int_0^x V_\eta(t)\dd t, \qquad \text{for a.e. }\eta \in S.
\]
If $L_{V_\eta}$ denotes the half-line operator corresponding to the restriction of $V_\eta$ to $[0,\infty)$, then $\E = \sigma_\ess(L_{V_\eta})$ almost surely, so as a direct consequence of our deterministic results, $\E$ corresponds to a Martin function with an expansion \eqref{Martinexp2term}, and
\begin{equation}\label{ineqaEergodic}
a_\E \le \bbE(V).
\end{equation}
This inequality is new; several cases of the equality $a_\E = \bbE(V)$ are well known and among the most studied classes of ergodic Schr\"odinger operators (periodic, reflectionless almost periodic with finite gap length), and we can now interpret this through the fact that the corresponding potentials are regular.

In the ergodic setting, two central objects is the Lyapunov exponent $\gamma(z)$ and the density of states $\dd \rho$; both are almost sure ergodic averages of important spectral quantities. The transfer matrix $T_\eta(x,z)$ is the $2\times 2$-matrix valued solution of the initial value problem
\[
(\partial_x T_\eta)(x,z) = \begin{pmatrix} 0 & V_\eta(x) - z \\ 1 & 0 \end{pmatrix} T_\eta(x,z), \qquad T_\eta(0,z) = I,
\]
and the corresponding Dirichlet solution is $u_\eta(x,z) = (T_\eta)_{2,1}(x,z)$. If $\rho_{\eta,x}$ denotes the measure corresponding to $u_\eta$ as in \eqref{eqn:rhoxdefn}, then
\begin{align}\label{eq:lyaponuv}
	\gamma(z) = \lim_{x\to +\infty} \frac 1x \log \lVert T_\eta(x,z) \rVert, \qquad \text{for a.e. }\eta \in S,	
\end{align}
and
\[
\dd\rho = \wlim\limits_{x\to +\infty} \dd\rho_{\eta,x}, \qquad \text{for a.e. }\eta \in S.
\]
Thus Theorem~\ref{thm14}, specialized to the ergodic setting, immediately gives the following:
\begin{corollary}\label{cor:ergodicregular}
For any ergodic family of Schr\"odinger operators obeying \eqref{eqn:ergodicbounded}, the following are equivalent:
\begin{enumerate}[(i)]
\item $a_\E = \bbE(V)$;
\item For every Dirichlet-regular $z\in \E$, $\gamma(z) = 0$;
\item For almost every $z\in \E$ with respect to harmonic measure, $\gamma(z) = 0$;
\item For all $z \in \bbC_+$, $\gamma(z) \le M_\E(z)$;
\item For all $z \in \bbC \setminus\E$, $\gamma(z) \le M_\E(z)$;
\item $\gamma(z) = M_\E(z)$ for all $z\in \bbC \setminus [\min \E, \infty)$.
\end{enumerate}
\end{corollary}

We say that a family of ergodic Schr\"odinger operators is regular if one (and therefore all) of the statements of Corollary~\ref{cor:ergodicregular} holds. Although this notion is new, let us point out that it contains several of the most well studied families of almost periodic Schr\"odinger operators.

For a $1$-periodic potential $V$, it is well known that the discriminant has an asymptotic expansion at $\infty$ whose coefficients are equal to averages of differential polynomials in $V$ (under the appropriate regularity assumptions on $V$). The first of those equalities, rewritten for the Martin function, give the equality $a_\E = \int_0^1 V(x)\dd x$. This can now be interpreted through the fact that periodic potentials are regular.

For an almost periodic potential $V$, Johnson--Moser \cite{JohMos82} introduced the spatial average of $m$-functions, whose real part is the Lyapunov exponent $\gamma$. Their construction relies heavily on almost periodicity through compactness of the hull, so their methods would not extend to our setting; \cite{JohMos82} noted as a consequence of their results, the spectrum of any almost periodic Schr\"odinger operator is not a polar set (i.e.\ $\Omega$ is Greenian), but further consequences of Theorem~\ref{thm:MainAkhiezerLevin} were not previously known even in the almost periodic case.

The next theorem is a specialization of Theorems~\ref{thm:dos1}, \ref{thm:dos2} to the ergodic setting:

\begin{theorem}
Let $(V_\eta)_{\eta\in S}$ be an ergodic family of Schr\"odinger operators obeying \eqref{eqn:ergodicbounded}. If this ergodic family is regular, then its density of states $\rho$ is equal to the Martin measure $\rho_\E$. Conversely, if $\rho = \rho_\E$, then either the ergodic family is regular, or for a.e. $\eta$, the maximal spectral measure $\mu_\eta$ is supported on a polar set.
\end{theorem}

Although positive Lyapunov exponents don't always correspond to localization, we can now prove that they always correspond to very thin spectral type. This is the analog of a Jacobi matrix result which has been described as the ultimate Pastur--Ishii theorem.

\begin{theorem}\label{thm:ultimatePasturIshii}
Let $\gamma$ denote the Lyapunov exponent associated to the ergodic family $(V_\eta)_{\eta\in S}$ and let $\mu_\eta$ denote a maximal spectral measure for $H_{V_\eta}$. Let $Q \subset \bbR$ be the Borel set of $\lambda \in \bbR$ with $\gamma(\lambda) > 0$. Then for a.e.\ $\eta\in S$, there exists a polar set $X_\eta$ such that $\mu_\eta(Q \setminus X_\eta) =0$. In particular, the measure $\chi_Q \dd\mu_\eta$ is of local Hausdorff dimension zero.
\end{theorem}

It is known in great generality \cite{DamSimSto02} that one-dimensional random Schr\"odinger operators give rise to positive Lyapunov exponent throughout the spectrum.  In particular, random Schr\"odinger operators provide examples of non-regular operators.

Throughout this paper, we follow the dominant literature by working with locally integrable potentials; we expect that the theory presented here can be extended to potentials which are in the negative Sobolev space $H^{-1}([0,x])$ for $x < \infty$, with an appropriate uniform bound replacing \eqref{L1locunif}.

\section{The Martin function and Akhiezer--Levin sets}\label{sec:Martinfunction}

In this section we consider in more detail the general Martin theory for Denjoy domains $\Omega=\C\setminus \E$ with $\min \E=b_0>-\infty$. Clearly, we have in mind the application that $\E$ is the essential spectrum of some  continuum Schr\"odinger operator, $L_V$, where $V$ satisfies \eqref{L1locunif}. 

Recall that the capacity of a Borel set $A$ is defined by $\Cap(A)=\sup\{\Cap(K):~ K~\text{compact,}~K\subset A\}$ and we call a Borel set, $A$, polar, if $\Cap(A)=0$. Moreover, a property holds quasi-everywhere on a set $B$, if there exists a polar set $A$ such that the property holds on $B\setminus A$. We start with a discussion of the Green function $G_\E(\z,\z_0)$, $\z_0\in\Omega$. For standard references on potential theory see \cite{ArmGar01,RansPotential,GarHarmonicMeasure}. If $\z_0\in\R$, then $G_\E(\z,\z_0)$ is symmetric, that is, $G_\E(\overline{\z},\z_0)=G_\E(\z,\z_0)$. Let  us fix $\z_0<b_0$. Then there exists a comb domain
\begin{align}
\Pi_{\z_0}=\{x+iy:\ 0<x<\pi,\ y>\height(x)\},
\end{align}
where $\height$ is a positive upper semicontinuous function, bounded from above, and vanishes Lebesgue-a.e., and a conformal mapping $\t_{\z_0}:\C_+\to \Pi_{\z_0}$ such that 
\begin{align}\label{eq:GreenTheta}
G_\E(\z,\z_0)=\Im\t_{\z_0}(\z),\quad \z\in\C_+.
\end{align}
(such a representation was proved in \cite{ErY12} in the case that $\E$ is compact and $\z_0=\infty$; by a simple transformation $\l=\frac{1}{\z_0-\z}$ this yields a corresponding representation for the current setting). 
Note that $\t_{\z_0}(b_0)=\limsup_{u\to 0}\height(u)$ and $\t_{\z_0}(\infty)=\limsup_{u\to \pi}\height(u)$. Moreover, the Lebesgue measure on the base of the comb corresponds to the  harmonic measure $\omega_{\E}(\cdot,z_0)$. The mapping can be extended by symmetry to $\C\setminus[b_0,\infty)$ such that \eqref{eq:GreenTheta} still holds there. In fact, any such function $s$ leads to a Green function of a certain domain. 

The Martin kernel normalized at $\z_*<b_0$ is defined on $\Omega\times (\Omega\setminus\{\z_*\})$ by
\begin{align}\label{eq:MartinKernel}
M_{\E}(\z,\z_0)=\frac{G_{\E}(\z,\z_0)}{G_{\E}(\z_*,\z_0)}.
\end{align}
The \textit{Martin compactification} $\widehat{\Omega}$ is the smallest metric compactification of $\Omega$ such that $M_\E(\z,\cdot)$ can be continuously extended to the boundary $\partial^M\Omega=\widehat{\Omega}\setminus \Omega$ for each $\z$. We will also write $M_\E(\z,\z_0)$ for the extended function. Note that by the Harnack principle the family $\{M_{\E}(\z,\z_0)\}$ is precompact in the space of positive harmonic functions equipped with uniform convergence on compacts. We call a positive harmonic function, $M$, \textit{minimal} if any harmonic function, $h$, which satisfies $0\leq h\leq M$, is a multiple of $M$, i.e.,  $h= cM$, $c\geq 0$. Finally, let $\partial^M_1\Omega\subset \partial^M\Omega$ denote the subset of the Martin boundary, which consists of minimal harmonic functions. In this case, for every positive harmonic function $h$, there exists a unique finite measure $\nu$ such that
\begin{align}\label{eq:RepPositHarm}
h(\z)=\int_{\partial^M_1\Omega}M_\E(\z,x)\dd\nu(x),\quad h(\z_*)=\nu(\partial^M_1\Omega).
\end{align}

In general $\partial^M_1\Omega$ can be quite abstract, but the situation is rather intuitive for Denjoy domains. In \cite[Theorem 6]{GardSjoed09} it is shown that there exists a map $\pi: \partial^M_1\Omega\to \E\cup \{\infty\}$ such that for every $x\in\E\cup\{\infty\}$,  $\#\pi^{-1}(\{x\})$ is either one or two, depending on how ``thin'' $\R\cap\Omega$ is at $x$. To state this precisely we need some definitions. If $A$ is a subset of the Martin boundary $\partial^M\Omega=\hat \Omega\setminus\Omega$, then we say a property, $P$, holds near $A$ if there is a Martin-neighborhood $A\subset W$ such that $P$ holds on $W\cap \Omega$. Then, for  $A\subset \hat\Omega$ and a superharmonic function $h$ on $\Omega$ we define the reduced function
\begin{align}\label{def:reduction}
R_h^A(x)=\inf\{u(x):\ u\geq 0\text{ is superharmonic, } h\leq u \text{ on } A\cap \Omega\text{ and }h\leq u \text{ near } A\cap \partial^M\Omega\}
\end{align}
and $\hat R_h^A$ denotes its lower semicontinuous regularization. A set $A\subset \Omega$ is said to be minimally thin at $y\in \partial^M_1\Omega$ if 
\begin{align*}
\hat R^A_{M_\E(\cdot,y)}\neq  M_\E(\cdot,y).
\end{align*}
Then $\#\pi^{-1}(\{x\})=2$ if and only if there is $y\in\pi^{-1}(\{x\})$  such that $\Omega\cap \R$ is minimally thin at $y$. Informally, if $\E$ is sufficiently ``dense'' at $x$, then $\Omega$ locally splits into the two half spaces $\C_+$ and $\C_-$ and we obtain a Martin function for each of them.

A reformulation of the above statement can be given in the following way. For $x\in\E$, let $\cP_\E(x)$ denote the set of positive harmonic functions that are bounded outside every neighborhood of $x$ and vanish quasi-everywhere on $\E$. As in the proof of \cite[Lemma 2.9]{Hirata07} one can see, that $\cP_\E(x)$ is spanned by the Martin functions related to $x$. Hence, the above question is whether $\cP_\E(x)$ is one- or two-dimensional. We will provide a simplified proof for the case that there is only one Martin function associated to $x$ below. This question has attracted much interest and several conditions have been obtained, \cite{Anc79,Be80,KoosisLogInt1,Lev89Par3}. To note two extreme cases, if $x\in (a,b)\subset\E$, then $\cP_\E(x)$ is two-dimensional, whereas if $x$ is a endpoint of a gap of $\E$, then $\cP_\E(x)$ is one-dimensional, as discussed in \cite{GardSjoed09} after Theorem 6. 

We are particularly interested in the Martin kernel related to $\infty$. Since $\E$ is semibounded, $\cP_\E=\cP_\E(\infty)$ is one-dimensional and we can talk about the Martin function $M_\infty(\z)=M_\E(\z,\infty)$ related to $\infty$ which is known to be symmetric, i.e., $M_\infty(\overline{\z})=M_\infty(\z)$. Moreover, all limits with $z_n\to-\infty$ must lead to $M_\infty$ and we have
\begin{align*}
M_\infty(z)=\lim\limits_{z_0\to-\infty}M(z,z_0)=\lim\limits_{z_0\to-\infty}\frac{\Im\theta_{z_0}(z)}{G_\E(z_*,z_0)}
\end{align*}
Note that $M_\infty$ is not exactly $M_\E$ from the introduction, because in the general situation we cannot use the normalization \eqref{AkhiezerLevin1}. For this reason, we keep the normalization at $z_*$, but once we have specified to sets where the limit in \eqref{AkhiezerLevin1} is positive, we can pass to this normalization. Since $M_\infty(\z)$ is positive and harmonic in $\Omega$, setting $\l^2=\z-b_0$ it defines a positive harmonic function for $\l\in\C_+$ by
$$
f(\l)=M_\infty(\z).
$$
Since $f$ can be represented  as
\begin{align}\label{eq:PosHarmonicFunc}
f(x+iy)=ay+\int\frac{y}{(x-t)^2+y^2}\dd \nu(t),\quad \int\frac{\dd\nu(t)}{1+t^2}<\infty
\end{align}
and
\begin{align}\label{eq:MassAtInfty}
0\leq a=\lim_{y\to\infty}\frac{f(iy)}{y},
\end{align}
we see that $M_\infty(\z)$ can grow at most as $\sqrt{-\z}$ as $\z\to -\infty$. In case of two-sided unbounded sets, where the Martin function can grow at most linearly, Akhiezer and Levin showed that $\cP_\E$ is two-dimensional whenever the Martin function admits the maximal possible growth. This explains why we call $\E$ an \textit{Akhiezer--Levin set} if
\begin{align}\label{eq:GrowthMartin}
\lim\limits_{z\to-\infty}\frac{M_\infty(z)}{\sqrt{-z}}>0.
\end{align}
Note that by \eqref{eq:MassAtInfty} this limit indeed exists in $[0,\infty)$. 
Since in \eqref{eq:PosHarmonicFunc}, $\int\frac{y}{(x-t)^2+y^2}\dd \nu(t)$ defines again a positive harmonic function it follows that 
\begin{align}\label{eq:MartinLowerBound}
a\Re\sqrt{b_0-z}\leq M_\infty(z)
\end{align}
in $\Omega$. 
The following theorem presents a list of equivalent characterizations of $M_\i$. We say that $h$ vanishes continuously at a point $x \in \E$ if $\lim_{\substack{z\to x\\ z \in \Omega}} h(z) = 0$. We call a subset of $\Omega$ bounded if it is bounded as a subset of $\bbC$.

\begin{theorem}\label{lem:ConePositive}
Let $H_{+,b}(\Omega)$ denote the set of positive harmonic functions on $\Omega$ that are bounded on every bounded subset of $\Omega$. Then, the following are equivalent:
\begin{enumerate}[(i)]
	\item $h\in H_{+,b}(\Omega)$ and $h$  vanishes continuously for every Dirichlet-regular point of $\E$; 
	\item $h\in H_{+,b}(\Omega)$ and $h$   vanishes continuously quasi-everywhere on $\E$;
	\item $h\in H_{+,b}(\Omega)$ and $h$   vanishes continuously $\omega_\E(\cdot,z_0)$-a.e.;
	\item $h=cM_\infty$, where $c\geq 0$;
\end{enumerate}
\end{theorem}
\begin{proof}
Due to \cite[Remark 5, Theorem 6]{GardSjoed09} $(iv)\implies(i)$. Kellogg's theorem \cite[Corollary 6.4]{GarHarmonicMeasure} yields $(i)\implies(ii)$ and by \cite[Theorem III.8.2]{GarHarmonicMeasure} we get that $(ii)\implies(iii)$. It remains to show that $(iii)\implies(iv)$.
 Let $E_n=\{\pi^{-1}(\{x\}): x\in\E,\ x<n\}$ and $E_n^\mathsf{c}=\partial_1^M\Omega\setminus E_n$. Let $R_h^{E_n}$ denote the reduction of $h$ (which is harmonic since $E_n\subset \partial_1^M\Omega$).  Since $h\in H_{+,b}(\Omega)$, $h$ is majorized by a constant in $U \cap \Omega$ where $U$ is a neighborhood of $E_n$ in $\hat\Omega$, so $R_h^{E_n}$ is a bounded harmonic function in $\Omega$ which vanishes $\omega_\E(\cdot,z_0)$-a.e. on the boundary. By the maximum principle \cite[Theorem 8.1]{GarHarmonicMeasure} it follows that $R_h^{E_n}=0$. Therefore,
\begin{align*}
h=R^{\partial_1^M\Omega}_h\leq R_h^{E_n}+R_h^{E_n^\mathsf{c}}=R_h^{E_n^\mathsf{c}}\leq h,
\end{align*}
where we used \cite[Lemma 8.2.2, Corollary 8.3.4]{ArmGar01}. That is, $h=R_h^{E_n^\mathsf{c}}$ for all $n$. We want to show that $\lim\limits_{n\to\infty}R_h^{\E_n^\mathsf{c}}=R^{\{\infty\}}_h$. By the definition of the reduction operator, we have $R^{\{\infty\}}_h\leq R_h^{\E_n^\mathsf{c}}$. Since for any open neighborhood, $W$, of $\infty$ in $\hat\Omega$, there exists $n$ such that $E_n^\mathsf{c}\subset W$, we obtain from \eqref{def:reduction} that $R_h^{\{\infty\}}=\lim\limits_{n\to\infty}R_h^{\E_n^\mathsf{c}}=h$. That is, in the language of potential theory, $\infty$ is a pole of $h$ and by \cite[Theorem 8.2.7]{ArmGar01}, $h=cM_\infty$.
\end{proof}

In his series of papers \cite{Lev89,Lev89part2,Lev89Par3}, Levin first systematically established the relation between extremal problems and comb mappings imposing Dirichlet-regularity on the set $\E$. Eremenko and Yuditskii \cite{ErY12} provided a modern approach to it, giving a detailed proof for comb mappings for Green functions as discussed above. It relies on the representation of Green functions for a compact set $E$, as
\begin{align}\label{eq:Green}
G_E(z,\infty)=\int_E\log|z-t|\dd\rho_E(t)+\g_E,
\end{align} 
where $\operatorname{Cap}(E)=e^{-\g_E}$ and $\rho_E(X)=0$ for sets of zero capacity. It is also discussed that the proof carry over for Martin functions and the corresponding description is given. Since we were not able to find in our generality a reference for a representation of the type \eqref{eq:Green}, which is certainly known to experts, for the readers convenience we survey the corresponding theory in the following. 

Since $M_\infty$ vanishes quasi-everywhere, we can extend $M_\infty$ to a subharmonic function to all of $\C$ by 
\begin{align}\label{eq:subharmExt}
M_\infty(x)=\limsup_{\substack{\z\to x \\ \z\in\Omega}}M_\infty(\z),\quad x\in\E,
\end{align}
see \cite[Theorem 5.2.1]{ArmGar01}. Hence, we obtain a subharmonic, symmetric function in $\C$, which is positive and harmonic in $\C_+$ and $\C_-$. For the following result we refer to \cite[Lemma 2.3]{Lev89part2} and its corollary. It was initially proved for majorants of subharmonic functions, but it is mentioned that it extends to the version stated below:
\begin{lemma}\label{lem:Levin1}
Let $v$ be a subharmonic, symmetric function in $\C$, which is positive and harmonic in $\C\setminus [b_0,\infty)$ for some $b_0\in\R$. Then 
\begin{align}\label{eq:HadamardRep}
v(\z)=v(z_*)+\int_{b_0}^\infty\log\left|1-\frac{\z-\z_*}{t-\z_*}\right|\dd\nu(t),\quad \int_{b_0}^\infty\frac{\dd\nu(t)}{t-\z_*}<\infty,
\end{align}
and for $y>0$ 
\begin{align}\label{eq:DerivNevanlinna}
\frac{\partial v(x+iy)}{\partial y}=\int_{b_0}^\infty\frac{y}{(t-x)^2+y^2}\dd \nu(t)>0.
\end{align}
\end{lemma}
\begin{remark}
\eqref{eq:HadamardRep} is essentially the Hadamard representation for the subharmonic function $v$ and $\nu$ is its \textit{Riesz measure}. Usually the Hadamard representation would include a normalization term $\frac{\Re z}{t}$, which is not needed due to the convergence property of $\nu$ in \eqref{eq:HadamardRep}.
 \end{remark}

\begin{lemma}\label{lem:theta}
Let $\Theta$ be such that $\Im \Theta= M_\infty$ for $z\in\C_+$ and $\rho$ be the Riesz measure for $M_\infty$. Then, the functions $\Theta$ and $i\Theta'$ are Herglotz functions and in particular
\begin{align*}
i\Theta'(\z)=\int_\E\frac{\dd\rho(t)}{t-\z}.
\end{align*}
They can be analytically extended to $\C\setminus[b_0,\infty)$ and $\Theta'\neq 0$ there.
\end{lemma}
\begin{proof}
Applying Lemma~\ref{lem:Levin1} to $M_\infty$ gives a representation of the form \eqref{eq:HadamardRep} in terms of the Riesz measure $\rho$ supported on $\E$ and, in particular, $\int_\E \frac{\dd \rho(t)}{t-z_*} < \infty$. Moreover,
\begin{align}\label{eq:tPrime1}
i\Theta'(\z)=c_0+\int_\E\frac{\dd\rho(t)}{t-\z}
\end{align}
for some $c_0\in\R$, since the imaginary parts of the two sides are equal by \eqref{eq:DerivNevanlinna}. Since $\Theta$ is also a Herglotz function, for some measure $\mu$ supported on $\E$,
\begin{align}\label{eq:Dec1}
i\Theta'(z)=i\int\frac{\dd\mu(t)}{(t-z)^2},\quad \int\frac{\dd\mu(t)}{1+t^2}<\infty.
\end{align}
Using monotone convergence and taking the limit as $z\to-\infty$ in \eqref{eq:tPrime1} and \eqref{eq:Dec1} yields $\lim_{z\to-\infty}i\Theta'(z)=0=ic_0$. 
Since $i\Theta'$ is Herglotz, $\Theta'\neq 0$ in $\C_+$ and $\C_-$. Moreover, since it is increasing on  $(-\infty,b_0)$ and vanishes at $-\infty$ we obtain the final claim.
\end{proof}

The following lemma shows that, like the harmonic measure, $\rho$ gives zero measure to polar sets. Of course, once we introduce the Martin measure $\rho_\E$, it will be a scalar multiple of $\rho$, so the following claim will also hold for $\rho_\E$.
\begin{lemma}\label{lem:rhoPolar}
	Let $X\subset \C$ be a Borel polar set. Then $\rho(X)=0$.
\end{lemma}
\begin{proof}
	By \cite[Theorem 3.2.3]{RansPotential} it suffices to show that for each $s>b_0$ we have
	\begin{align}\label{eq:Nov27}
		\int_{b_0}^s\int_{b_0}^s\log|x-t|\dd\rho(x)\dd\rho(t)>-\infty
	\end{align}  
	By means of the subharmonic extension \eqref{eq:subharmExt}, $M_\i$ is non-negative on $\C$ and we get 
	\begin{align*}
		0\leq\int_{b_0}^{s}M_\infty(x)\dd\rho(x)=d+I_1+I_2,
	\end{align*}
	where
	\begin{align*}
	d&=\rho(b_0,s)\left(1-\int_{b_0}^{s}\log|t-\z_*|\dd\rho(t)\right),\quad I_1=\int_{b_0}^{s}\int_{b_0}^{s}\log|x-t|\dd\rho(t)\dd\rho(x),\\ I_2&=\int_{b_0}^{s}\int_{s}^{\infty}\log\left|1-\frac{x-\z_{*}}{t-\z_*}\right|\dd\rho(t)\dd\rho(x).
	\end{align*}
	Since $I_2\leq 0$, it follows that $-\infty<-d\leq I_1$, i.e., we have \eqref{eq:Nov27}.
\end{proof}

It was already encountered in \cite[Lemma 2.4]{Lev89part2} that there is an explicit connection between $\rho$ and the conformal map $\Theta$ defined in Lemma \ref{lem:theta}, see also \cite{ErY12}. Note that although in \cite{Lev89part2} Dirichlet-regularity is assumed for the set $\E$ the proof of the following lemma holds also in our setting. Namely, the Lebesgue measure on the base of the comb corresponds to the measure $\rho$ on $\E$. To be more precise, $\Re\Theta$ extends continuously to $\R$ and we have 
\begin{align}\label{eq:thetaMeasure}
\Re\Theta(b)-\Re\Theta(a)=\pi \rho((a,b)).
\end{align}

These are all the ingredients needed to describe the comb domains related to the conformal mapping $\Theta$. There exists a positive upper semicontinuous function $\height$ on $(0,b)$, where $b\in(0,\infty]$ such that $\Theta$ maps $\C_+$ conformally onto
\begin{align*}
	\Pi=\{x+iy:\ 0<x<b, y>\height(x)\}.
\end{align*}
If $b<\infty$ then $\limsup_{x\to b}\height(x)=\infty$. We will show in Corollary \ref{cor:InftyDirichletRegular} that $b$ being finite corresponds to $\infty$ being not Dirichlet-regular. 
\begin{example}
In their classical work \cite{MarOst75} Marchenko--Ostrovskii studied the relation between spectra of 1-periodic $L^2$ potentials on the real line and corresponding data of the mapping $\Theta_\E$. They showed that $\E$ is the spectrum of a Schr\"odinger operator of this type if and only if the corresponding comb domain is of the form 
\begin{align*}
\Pi_\E=\{x+iy, x>0, y>0\}\setminus\{k\pi+iy: k\in \bbN, 0\leq y\leq \height_k\},
\end{align*}
and the slit heights $\height_k$ satisfy $\sum_{k=1}^{\infty}k^2\height_k^2<\infty$. 
\end{example}

The next example will construct examples of Akhiezer--Levin sets which don't have an expansion of the form \eqref{Martinexp2term}.

\begin{example}
We will construct an explicit expression for the conformal map, $\Theta:\C_+ \to \Pi=\C_+\setminus \{n+iy: n\in \Z, 0<y< y_0\}$, where $y_0>0$ is an arbitrary but fixed parameter. We will show that along the imaginary axis we have
\begin{align*}
\Theta(iy)=iy+ic(y_0)+o(1),\quad \text{as }y\to\infty,
\end{align*}
where, $c(y_0)$ is a real constant that depends monotonically on $y_0$ and can attain in fact any real value.
Note that $\Theta$ can be continuously extended to $\R$ and that $E:=\Theta^{-1}(\R)$ is symmetric, $E=-E=\{-x, x\in E\}$. Hence, again by $\tilde \Theta(z)=\Theta(\l^2)$, $M(z)=\Im \tilde \Theta(z)$ provides an example for a Martin function of an Akhiezer--Levin set, which has a constant term in its asymptotic expansion. The Christoffel--Darboux transformation 
\begin{align*}
f_1(w)=\frac{1}{\pi}\int_{-1}^w\frac{\dd x}{\sqrt{1-x^2}}
\end{align*}
maps $\C_+$ onto $\Pi_1=\{\vartheta=\xi+i\eta: \eta>0, 0<\xi<1\}$. In particular $f_1(-1)=0$ and $f_1(1)=1$. We choose $\ell > 1$ so that $iy_0 = f_1(-\ell)$ and consider
\begin{align*}
f_2(w)=\frac{1}{\pi}\int_{-\ell}^w\frac{\dd x}{\sqrt{\ell^2-x^2}}=f_1(w/\ell).
\end{align*}
Then $\Theta = f_1\circ f_2^{-1}$ defines a conformal map $\Theta:\Pi_1\to\Pi_1$ such that $\Theta(0)=iy_0$. By symmetry, we can extend $\Theta$ to a conformal map from $\Theta:\C_+\to\Pi$. Calculations of $f_1, f_2$ along the imaginary axis give
$
\Theta(iy)=i\cosh^{-1}(\ell\cosh(y)),
$ 
so
\begin{align*}
\Theta(iy)=iy+i\log(\ell)+o(1),\quad \text{as }y\to\infty.
\end{align*}
\end{example}

We emphasize that in order to show that the limit in \eqref{eq:GrowthMartin} is always finite for the Martin function, it was only used that $M_\infty$ represents a positive harmonic function in $\Omega$. This shows that the same conclusion holds for any such function. In view of \eqref{eq:RepPositHarm} this growth should also be reflected in the corresponding asymptotic behavior of $M_\infty$, leading to the following criterion for $\E$ to be an Akhiezer--Levin set. 
\begin{lemma}\label{lem:AkhiezerLevin}
Assume that there exists a positive harmonic function in $\Omega$ such that 
\begin{align*}
\lim\limits_{z\to-\infty}\frac{h(z)}{\sqrt{-z}}=1.
\end{align*}
Then $\Omega$ is Greenian and $\E$ is an Akhiezer--Levin set. Moreover, in this case we have
\begin{align}\label{eq:lowerBound}
M_\E(\z)\leq h(\z),
\end{align}
for all $\z\in\Omega$, where $M_\E$ is normalized by $\lim_{z\to-\infty}M_\E(\z)/\sqrt{-z}=1$.
\end{lemma}
\begin{proof}
By Myrberg's theorem \cite[Theorem 5.3.8]{ArmGar01} the existence of a non-constant positive harmonic function on $\Omega$ implies that $\Omega$ is Greenian. Since $h$ is a positive harmonic function in $\Omega$ there exists a unique measure $\nu$ with $\nu(\partial^M_1\Omega)=h(\z_*)$ such that
\begin{align*}
h(\z)=\int_{\partial^M_1\Omega}M(\z,x)\dd\nu(x).
\end{align*}
In particular, $\nu(\{\infty\})<\infty$. Recall that $\# \pi^{-1}(\{\infty\}) = 1$. Since $(-\infty,b_0)\subset\Omega$, the negative half axis is clearly not minimally thin at $\infty$ so it follows by \cite[Theorem 9.2.6]{ArmGar01} that 
\begin{align}\label{eq:Nov19}
\liminf_{z\to-\infty}\frac{h(\z)}{M_\infty(\z)}\leq\nu(\{\infty\})<\infty.
\end{align}
Let $\l^2=\z-b_0$ and $g(\l)=h(\z)$ and $f(\l)=M_\infty(\z)$. Then $f$ defines a positive harmonic function in $\C_+$ and 
$$
f(x+iy)=ay+\int\frac{y}{(x-t)^2+y^2}\dd \mu(t),\quad a=\lim\limits_{y\to\infty}\frac{f(iy)}{y}.
$$
Hence,
\begin{align*}
0<\limsup_{\z\to-\infty}\frac{M(\z)}{h(\z)}=\limsup_{y\to\infty}\frac{f(iy)}{g(iy)}=\limsup_{y\to\infty}\frac{f(iy)}{y}=a.
\end{align*}
Hence, $\E$ is an Akhiezer--Levin set. Due to \cite[Theorem 9.3.3]{ArmGar01} we have
\begin{align}\label{eq:Nov19_2}
\nu(\{\infty\})=\inf_{\z\in\Omega}\frac{h(\z)}{M_\infty(\z)}\leq \frac{h(\z)}{M_\infty(\z)}.
\end{align}
and the second claim follows. Finally, \eqref{eq:Nov19_2} shows that we actually have equality in $\eqref{eq:Nov19}$ and it follows that $\nu(\{\infty\})$ corresponds to the normalization of $M_\infty$ at $\infty$. 
\end{proof}

Carleson and Totik \cite{CarTot04} showed that $\cP_\E(x_0)$ being two-dimensional is equivalent to the fact that $G(\z,z_0)$ is Lipschitz continuous at $x_0$, where $z_0$ is some arbitrary interior point. As a corollary of the comb mapping representation for $\Theta$, we show that $\E$ being an Akhiezer--Levin set implies continuity at infinity. Note that by the aforementioned equivalence, one cannot hope for Lipschitz continuity for semibounded sets, since in this case $\cP_\E(\infty)$ is always one-dimensional. Alternatively, this could be seen from the fact that often, at a gap edge $a$, the Green function has behaviour $G(z,z_0)\sim \sqrt{z-a}$ and thus is not Lipschitz continuous. Moreover, as discussed in \cite{VoYud16} the set $\E=\R_+\setminus\cup_{n\in\Z}r^n(a_1,b_1)$, where $0<a_1<b_1$ and $r>1$ provides an example of a set for which $\infty$ is Dirichlet-regular, but which is not an Akhiezer--Levin set. In this sense the following result is optimal. 
\begin{corollary}\label{cor:InftyDirichletRegular}
Let $\E\subset \R$ be closed and semibounded and $\Theta$ the corresponding comb-mapping. If $\sup\{\Re\Theta(z): z\in\C_+\}=\infty$, then $\infty$ is a Dirichlet-regular point of $\E$. This holds in particular if $\E$ is an Akhiezer--Levin set.
\end{corollary}
\begin{proof} 
We will assume that $\limsup_{z_0\to-\infty}G(z_0,\z_*)=\e>0$ in order to obtain a contradiction. Note that $\sup\{\Re \theta_{z_0}(z):\ z\in\C_+\}=\pi$, so for any $z \in \bbC_+$,
\begin{align*}
\lim\limits_{z_0\to-\infty} \frac{\Re \theta_{z_0}(z)}{G(z_*,z_0)}\leq \liminf\limits_{z_0\to-\infty} \frac{\pi}{G(z_*,z_0)}.
\end{align*}
Since $\Theta(z) = \lim_{z_0\to-\infty} \frac{\theta_{z_0}(z)}{G(z_*,z_0)}$, taking the supremum over $z\in\C_+$ gives
$
\sup\{\Re \Theta(z):\ z\in\C_+\} \leq \e^{-1}\pi < \infty.
$
Now, as already mentioned in \cite{ErY12}, using upper semicontinuity of $h$ it follows that vanishing of the radial limit of $G(z_0,\z_*)$ implies Dirichlet-regularity. Let $\Im\theta_{z_*}=G(z,z_*)$ and it will be more convenient to shift the mapping by $-\pi$. Then,  $\lim_{z_0\to-\infty}G(z_0,\z_*)=0$ implies that $\limsup_{u\to 0}h(u)=0$. Therefore, $(-\infty,\z_*)$  is mapped by $\theta_{z_*}$ onto $i\R_+$ and we can extend $\theta_{z_*}$ by symmetry to $\C\setminus(\R\setminus(-\infty,z_*))$. In particular $i\R_+$ is an interior ray of the image, $\Pi_e=\Pi_{z_*}\cup i\R_+ \cup \{-x+iy: x+iy \in\Pi_{z_*}\}$, of this extended map. $\limsup_{u\to 0}h(u)=0$ now implies that $\Pi_e$ is locally connected at $0$ and hence, $\theta_{z_*}$ can be continuously extended to $0$ which implies that $\infty$ is a Dirichlet-regular point. This finishes the proof of the first claim.

In view of \eqref{eq:thetaMeasure}, $\sup\{\Re\Theta(z): z\in\C_+\}<\infty$ means that $\rho$ is finite. We show that this implies that $M_\infty$ can grow at most like $\rho(\R)\log|z|$ and therefore $\E$ is not an Akhiezer--Levin set. Let's assume that $|z_*-b_0|>1$ and $z_*<0$. Then, using \eqref{eq:HadamardRep} we see that for $z<z_*$ we have 
\begin{equation*}
M_\infty(z)-\rho(\R)\log|z|=M_\infty(z_*)+\int_{b_0}^\infty\log\left|\frac{1}{t-\l_*}\left(1-\frac{z_*}{z}\right)\right|\dd\rho(t)\leq M_\infty(z_*). \qedhere
\end{equation*}
\end{proof}

For Akhiezer--Levin sets one could also use the result of Carleson and Totik and the substitution $\l^2=z-b_0$ to see that $G$ is H\"older continuous with exponent $1/2$ at $\infty$.


\section{Asymptotic behavior of eigensolutions}

We now turn our attention to the Schr\"odinger operator $L_V$ and associated objects. Fundamental solutions at $z\in \C$ are defined as solutions $u(x,z), v(x,z)$ of the initial value problems
\begin{align}
-\partial_x^2 u + (V(x) -z) u =0,&  \qquad u(0,z)=0, \qquad (\partial_x u)(0,z) = 1  \label{defnu} \\
-\partial_x^2 v + (V(x) -z) v =0, & \qquad v(0,z)=1, \qquad (\partial_x v)(0,z) = 0 \label{defnv}
\end{align}
The natural regularity class for the solutions are functions which are in $W^{2,1}([0,x])$ for every $x < \infty$, and the differential equations are interpreted as equality of $L^1$ functions, i.e., equality Lebesgue-a.e.. It is useful to substitute
\[
k = \sqrt{-z}
\]
and view the initial value problems as perturbations by $V$ of $-\partial_x^2+k^2$. We will always assume that $\Re k \ge 0$; this can be done pointwise throughout $\C$, and later we will view $k$ as a branch of the square root such that $\Re k > 0$ if $z \in \C \setminus [0,\infty)$. Note also that this makes $\Im k < 0$ if $z \in \C_+$. By choosing the branch $\sqrt{z}=ik$, we see that $\sqrt{z}\in \C_+$ if $z\in \C \setminus [0,\infty)$. In particular, $\Im \sqrt{z}=\Re k$.

The fundamental solutions for $V=0$ are the functions
\[
c(x,k) = \cosh(kx), \qquad s(x,k) = \begin{cases} \frac{ \sinh(kx)}{k} & k \neq 0 \\ x & k = 0 \end{cases}.
\]
By standard arguments,  for general $V\in L^1([0,1])$, the initial value problems \eqref{defnu}, \eqref{defnv} are rewritten as integral equations, and by Volterra-type arguments, convergent series representations are then found for the fundamental solutions. With the notation
$\Delta_n(x) = \{ t \in \R^n \mid x \ge t_1 \ge t_2 \ge \dots \ge t_n \ge 0 \}$, the series representations for fundamental solutions and their first derivatives are 
\begin{align}
u(x,z) & = s(x,k) + \sum_{n=1}^\infty \int_{\Delta_n(x)} s(x-t_1,k) \left( \prod_{j=1}^{n-1} V(t_j) s(t_{j}-t_{j+1},k) \right) V(t_n)  s(t_n,k) \dd^n t \label{fundamentalseries1} \\
v(x,z) & = c(x,k) + \sum_{n=1}^\infty \int_{\Delta_n(x)} s(x-t_1,k) \left( \prod_{j=1}^{n-1} V(t_j) s(t_{j}-t_{j+1},k) \right) V(t_n)  c(t_n,k) \dd^n t \label{fundamentalseries2} \\
(\partial_x u)(x,z) & = c(x,k) + \sum_{n=1}^\infty \int_{\Delta_n(x)} c(x-t_1,k) \left( \prod_{j=1}^{n-1} V(t_j) s(t_{j}-t_{j+1},k) \right) V(t_n)  s(t_n,k) \dd^n t \label{fundamentalseries3}\\
(\partial_x v)(x,z) & = k^2 s(x,k) + \sum_{n=1}^\infty \int_{\Delta_n(x)} c(x-t_1,k) \left( \prod_{j=1}^{n-1} V(t_j) s(t_{j}-t_{j+1},k) \right) V(t_n)  c(t_n,k) \dd^n t \label{fundamentalseries4}
\end{align}
These expansions are derived, e.g., in \cite{PoschlTrubowitz87} for $V\in L^2([0,x])$, but they hold for $V\in L^1([0,x])$ as well, due to the estimate
\begin{equation}\label{simplexestimate}
\left\lvert \int_{0}^{x} \int_{0}^{t_1} \dots \int_{0}^{t_{n-1}}  e^{ \Re k (x -t_1)}  \left( \prod_{j=1}^{n} V(t_j) e^{ \Re k (t_{j} - t_{j+1})}  \right) V(t_n) e^{ \Re k  t_n} \, dt_n \dots dt_2\, dt_1  \right \rvert  \le \frac 1{n!} \left( \int_0^x \lvert V(s) \rvert \dd s \right)^n e^{ \Re k  x}
\end{equation}
which is proved by combining the exponentials and using permutations of $t$ and symmetry, and the elementary estimates which follow directly from Euler's formula,
\begin{equation}\label{eqn:Euler}
\lvert c(x,k) \rvert  \le e^{ \Re k  x},  \qquad
\lvert s(x,k) \rvert  \le \lvert k\rvert^{-1}  e^{ \Re k  x}.
\end{equation}
The same estimates which guarantee convergence, provide exponential upper bounds on eigensolutions; these are often stated over a fixed interval, but we will need a kind of uniformity in $x$:

\begin{lemma}
For all $z = -k^2 \in \bbC$ and $x > 0$,
\begin{equation}\label{25nov5}
\lvert u(x,-k^2) \rvert \le e^{(1+\Re k) x + \int_0^x \lvert V(t) \rvert dt}.
\end{equation}
\end{lemma}

\begin{proof} 
Using $\lvert s(x,k) \rvert = \lvert \int_0^x c(t,k) \dd t \rvert \le x e^{\Re k x} \le e^{(1+\Re k)x}$  and then applying \eqref{simplexestimate} to each term of \eqref{fundamentalseries1} implies that
\[
\lvert u(x,-k^2) \rvert  \le  e^{(1+\Re k) x} \sum_{n=0}^\infty \frac 1{ n!} \left( \int_0^x \lvert V(t) \rvert \right)^n. \qedhere
\]
\end{proof}

\begin{corollary}\label{cor:uniformUpperBound}
If $V$ obeys \eqref{L1locunif}, for each $R>0$ there exists $C_R$ such that for all $\lvert z\rvert\le R$ and $x \ge 1$ we have $\frac 1x \log \lvert u(x,z) \rvert \le C_R$.
\end{corollary}

\begin{proof}
This is an immediate consequence of the previous lemma together with $\int_0^x \lvert V(t) \rvert \dd t \le C (x+1) \le 2 Cx$ for $x \ge 1$, where $C = \sup_{x \ge 0} \int_{x}^{x+1} \lvert V(t) \rvert \dd t$.
\end{proof}

We will need asymptotic statements about $m$-functions. Such statements are ubiquitous, especially for smooth potentials; we need an asymptotic expansion which doesn't assume any smoothness.

\begin{lemma}\label{lemmaexpansionm}
For fixed $x > 0$, as $z \to \infty$, $\arg z \in [\delta,2\pi - \delta]$, 
\[
- \frac{v(x,z)}{u(x,z)} = - k - \int_0^x V(t) e^{-2kt}\,dt + \frac 1k \int_0^x \int_0^{t_1} e^{-2kt_1} (1 -e^{-2kt_2}) V(t_1) V(t_2)\,dt_2\,dt_1 + O(\lvert k\rvert^{-2})
\]
uniformly in $V$ in bounded subsets of $L^1([0,x])$.
\end{lemma}

\begin{proof}
Assume that $\int_0^x \lvert V(t) \rvert \dd t \le C$. Denote
\[
A_n = 2 k^{n+1} e^{-kx} \int_{\Delta_n(x)} s(x-t_1,k) \left( \prod_{j=1}^{n-1} V(t_j) s(t_{j}-t_{j+1},k) \right) V(t_n)  s(t_n,k) \, d^n t,
\]
\[
B_n = 2 k^n e^{-kx} \int_{\Delta_n(x)} s(x-t_1,k) \left( \prod_{j=1}^{n-1} V(t_j) s(t_{j}-t_{j+1},k) \right) V(t_n)  c(t_n,k) \, d^n t,
\]
From \eqref{eqn:Euler} and \eqref{simplexestimate} it follows that $\lvert A_n \rvert, \lvert B_n \rvert \le \frac{2 C^n}{n!}$. In the nontangential limit $z \to \infty$, $\arg z \in [\delta,2\pi- \delta]$, we have the elementary estimates
\[
\frac{ s(x,k) }{\frac{ e^{kx} }{2k} } = 1 - e^{-2kx} = 1 + O(\lvert k\rvert^{-3}), \quad \frac{ c(x,k) }{\frac{ e^{kx} }{2} } = 1 + e^{-2kx} = 1+ O(\lvert k\rvert^{-3}), 
\]
so the series expansions for $u(x,z)$, $v(x,z)$ imply
\begin{align*}
u(x,z) & = \frac{e^{kx}}{2k} \left( 1 + \frac{A_1}k + \frac{A_2}{k^2} +O(\lvert k\rvert^{-3}) \right), \\
v(x,z) & = \frac{e^{kx}}{2} \left( 1 + \frac{B_1}k + \frac{B_2}{k^2} + O(\lvert k\rvert^{-3}) \right),
\end{align*}
with the error $O(\lvert k\rvert^{-3})$ depending only on $C$ and $\delta$. Dividing,
\begin{equation}\label{29nov3}
- \frac{v(x,z)}{u(x,z)} 
= -k \left( 1 + \frac{B_1 - A_1}k + \frac{B_2 - A_2 - A_1(B_1 - A_1)}{k^2} + O(\lvert k\rvert^{-3}) \right).
\end{equation}
Moreover,
\begin{equation}\label{29nov1}
B_1 - A_1 = \int_0^x (1-e^{-2k(x-t)})V(t) e^{-2kt}\,dt = \int_0^x V(t) e^{-2kt}\,dt + O(e^{-2\Re k x}) 
\end{equation}
Multiplying by $A_1 = \frac 12 \int_0^x (1-e^{-2k(x-s)}) V(s) (1-e^{-2ks}) \,ds$ gives a formula for $A_1(B_1-A_1)$ as a double integral over $[0,x]^2$, and using the substitution $t_1 = \max\{s,t\}$, $t_2 = \min\{s,t\}$ gives
\begin{align*}
A_1 (B_1 - A_1) 
 & = \frac 12 \int_0^x \int_0^{t_1} (e^{-2kt_1} + e^{-2kt_2}  - 2 e^{-2k(t_1+t_2)} - e^{-2k(x-t_1+t_2)}  ) V(t_1) V(t_2) \,dt_2 \,dt_1 + O( e^{-2\Re k x})
\end{align*}
(some terms are grouped into the error $ O( e^{-2\Re k x})$ since, e.g., $x-t_2+t_1 \ge x$). Similarly,
\begin{align*}
B_2 - A_2 & = \frac 12 \int_0^x \int_0^{t_1} (1 - e^{-2k(x-t_1)}) V(t_1) (1 - e^{-2k(t_1 -t_2)}) V(t_2) e^{-2kt_2}\,dt_2 \,dt_1 \\
& = \frac 12 \int_0^x \int_0^{t_1} (e^{-2kt_2} - e^{-2kt_1} - e^{-2k(x-t_1+t_2)} ) V(t_1) V(t_2) \,dt_2 \,dt_1 + O(e^{-2\Re k x})
\end{align*}
Substituting these formulas into \eqref{29nov3} concludes the proof.
\end{proof}

Returning to the half-line setting from the introduction, we recall that half-line potentials obeying the boundedness assumption \eqref{L1locunif} are in the limit point case at $+\infty$, i.e., for every $z\in \C \setminus \E$, the set of solutions of
\[
-\partial_x^2 \psi + V \psi = z \psi, \qquad \psi \in L^2((0,\infty))
\]
is one-dimensional. Any such nontrivial solution is called the Weyl solution; it is uniquely determined up to normalization and we will not fix any particular normalization. We will use
\begin{align}\label{eq:mWeylfunction}
m(x,z) = \frac{(\partial_x \psi)(x,z)}{\psi(x,z)}.
\end{align}

\begin{proposition}\label{propexpansionm}
As $z \to \infty$, $\arg z \in [\delta,\pi - \delta]$,
\[
m(s,z) = - k - \int_0^1 V(s+t) e^{-2kt}\,dt + \frac 1k \int_0^1 \int_0^{t_1} e^{-2kt_1} (1 -e^{-2kt_2}) V(s+t_1) V(s+t_2)\,dt_2\,dt_1 + O(\lvert k\rvert^{-2})
\]
and the error is uniform in $s \in [0,\infty)$ if $V$ obeys \eqref{L1locunif}.
\end{proposition}

\begin{proof}
By an argument of Atkinson \cite{Atkinson81}, for $\arg z \in [\delta,\pi-\delta]$, the Weyl circle at $x$ has radius 
\[
r = \frac{2\lvert k \rvert^2}{\lvert \Im k \rvert} e^{-2x\Re k} (1+ O(\lvert k\rvert^{-1}))
\]
which decays exponentially as $z \to \infty$, $\arg z\in [\delta,\pi - \delta]$; the error term $O(\lvert k\rvert^{-1})$ is uniform for $V$ in bounded subsets of $[0,x]$, since this term is derived by arguments like those in the proof of Lemma~\ref{lemmaexpansionm}. Since $m_+(0,z)$ lies inside the Weyl circle and $-v(1,z) / u(1,z)$ lies on the circle, this radius allows us to estimate 
\[
\left\lvert m(0,z) + \frac{v(1,z)}{u(1,z)} \right\rvert \le  \frac{4\lvert k \rvert^2}{\lvert \Im k \rvert} e^{-2 \Re k} (1+ O(\lvert k\rvert^{-1})).
\]
In the nontangential limit as $\arg z \in [\delta,\pi-\delta]$, this error is $O(\lvert k\rvert^{-2})$, so the previous lemma implies
\[
m(0,z) = - k - \int_0^1 V(t) e^{-2kt}\,dt + \frac 1k \int_0^1 \int_0^{t_1} e^{-2kt_1} (1 -e^{-2kt_2}) V(t_1) V(t_2)\,dt_2\,dt_1 + O(\lvert k\rvert^{-2}).
\]
Applying this for an arbitrary $s \ge 0$ to the translated half-line potential $V_s(x) = V(x+s)$ on $[0,\infty)$ concludes the proof.
\end{proof}

For the half-line operator $L_V$, the Dirichlet solution can be interpreted as the Weyl solution corresponding to the endpoint $0$. Therefore, the Atkinson argument can be applied also ``in reverse'', to produce uniform asymptotics on the logarithmic derivative of $u(x,z)$. To produce uniform asymptotics, we fix the interval length $1$, as in the previous proof:
 
\begin{corollary}\label{corexpansionmminus}
As $z \to \infty$, $\arg z \in [\delta,\pi - \delta]$, for all $s \ge 1$,
\[
- \frac{(\partial_x u)(s,z)}{u(s,z)} = - k - \int_0^1 V(s-t) e^{-2kt}\,dt + \frac 1k \int_0^1 \int_0^{t_1} e^{-2kt_1} (1 -e^{-2kt_2}) V(s-t_1) V(s-t_2)\,dt_2\,dt_1 + O(\lvert k\rvert^{-2})
\]
and the error is uniform in $s \in [1,\infty)$ if $V$ obeys \eqref{L1locunif}.
\end{corollary}

To make some uniform statements for a family of Herglotz functions, we will use the Carath\'eodory inequality for the half-plane \cite[Proof of Theorem~I.8]{Lev80}: for any Herglotz function $f$,
\begin{equation}\label{eqn:Caratheodory}
\lvert f(z) \rvert \le \lvert f(i) \rvert + \Im f(i) \frac{2 \lvert z - i \rvert}{\lvert z+i\rvert - \lvert z - i\rvert}, \qquad \forall z\in \bbC_+.
\end{equation}

\begin{lemma}\label{lem:boundedWeyl}
Fix a potential $V$ which obeys \eqref{L1locunif}. For each $z \in \C_+$,
\begin{equation}\label{eqn:29nov9}
\sup_{x\ge 1} \left\lvert  \frac{(\partial_x u)(x,z)}{u(x,z)} \right\rvert <\infty.
\end{equation}
\end{lemma}

\begin{proof}
The ratio $- (\partial_x u)(x,z) / u(x,z)$ is a Herglotz function and obeys the nontangential asymptotics in Corollary~\ref{corexpansionmminus}. The error is uniform in $x \ge 1$ since $V$ obeys \eqref{L1locunif}. In particular, for $z = iy_0$ with some fixed $y_0 > 0$ large enough, Corollary~\ref{corexpansionmminus} implies an upper bound independent of $x$ and therefore \eqref{eqn:29nov9}. By rescaling by $y_0$ and using \eqref{eqn:Caratheodory}, the upper bound at $iy_0$ implies uniform upper bounds for $z$ in compact subsets of $\bbC_+$.
\end{proof}

For $z\notin\sigma(L_V)$, $\psi$ decays exponentially as $x \to \infty$. The Weyl solution $\psi$ and the Dirichlet solution $u$ are related by the Wronskian
\[
W(\psi,u) = (\partial_x u)(x,z) \psi(x,z) -  (\partial_x\psi)(x,z) u(x,z)
\]
which is independent of $x$ and nonzero, since $u$, $\psi$ are linearly independent (otherwise they would give an eigenvalue of $L_V$). 
This strongly suggests that $u$ should grow at the same rate at which $\psi$ decays, but a proof based only on the Wronskian is difficult due to the derivative, especially if a pointwise statement is desired. We therefore use a different argument:

\begin{lemma}\label{lemmaDirichletToWeyl}
Fix a potential $V$ which obeys \eqref{L1locunif}. For each $z \in \C_+$, there exists $C$ such that for all $x \in [1,\infty)$,
\[
C^{-1} \le \lvert u(x,z) \psi(x,z) \rvert \le C.
\]
\end{lemma}

\begin{proof}
We use the diagonal (spectral theoretic) Green's function for $L_V$,
\begin{equation}\label{diagonalGreenFunction}
g(x,x;z) = \frac{u(x,z) \psi(x,z)}{W(\psi,u)},
\end{equation}
which can be written as
\begin{equation}\label{23dec1}
- \frac 1{g(x,x;z)} = \frac{(\partial_x\psi)(x,z)}{\psi(x,z)} - \frac{(\partial_x u)(x,z)}{u(x,z)}.
\end{equation}
Using the above asymptotics for $m$-functions gives a well known asymptotic statement,
\[
g(x,x;z) = \frac 1{2\sqrt{-z}} + O(\lvert z\rvert^{-1}), \qquad z\to \infty, \arg z \in [\delta,\pi- \delta],
\]
and the proof given here shows that this asymptotic behavior is uniform in $x \in [1,\infty)$, since $V$ obeys \eqref{L1locunif}. In particular, for some fixed $z=iy$ with $y$ large enough, this implies $\sup_{x\in [1,\infty)} \lvert g(x,x;iy) \rvert < \infty$ and $\inf_{x\in [1,\infty)} \lvert g(x,x;iy) \rvert  > 0$. Rescaling $z$ by a factor $y$ and applying \eqref{eqn:Caratheodory} to the Herglotz functions $g(x,x;z)$ and $-1/g(x,x;z)$ implies uniform upper and lower bounds on compact subsets of $\bbC_+$.

For any $z \in \bbC_+$, the Wronskian is nonzero and independent of $x$, so by \eqref{diagonalGreenFunction}, uniform bounds in $x$ for $g(x,x;z)$ imply uniform bounds in $x$ (for each $z\in \C_+$) for $u(x,z) \psi(x,z)$. 
\end{proof}

The growth rate of $u(x,z)$ can now be expressed in terms of averages of the $m$-functions:

\begin{corollary}\label{corollaryDirichletToWeyl} 
For any $z \in \C_+$,
\begin{equation}
\limsup_{x\to\infty} \left\lvert \frac{1}{x}\log u(x,z) + \frac{1}{x}\int_{0}^{x}m(s,z)\dd s \right\rvert=0.
\end{equation}
\end{corollary}

\begin{proof}
This follows from Lemma \ref{lemmaDirichletToWeyl} since $m(x,z)$ is the logarithmic derivative of $\psi(x,z)$.
\end{proof}

 Expansions for $m(s,z)$ are often stated in terms of values of $V$ and its derivatives at $s$, but such expansions assume some regularity of $V$, and the error terms in such expansions are usually not uniform in the appropriate local norm for $V$. By working directly with the expansion in  Prop.~\ref{propexpansionm}, we can obtain uniform expansions for the averages without imposing any regularity on $V$.

\begin{corollary}\label{cor:maveragetoVaverage}
 If $V$ obeys \eqref{L1locunif},
\begin{equation}\label{30nov3}
\limsup_{x\to\infty} \left\lvert \frac 1{x} \int_0^{x} m(s,z)\dd s + k + \frac 1{2k x} \int_0^{x} V(s)\dd s \right\rvert =  O(\lvert k\rvert^{-2}),
\end{equation}
as $z = -k^2 \to \infty$, $\arg z \in [\delta,\pi-\delta]$, for any $\delta > 0$.
\end{corollary}

\begin{proof}
Due to the uniformity of the error in the asymptotic expansion from Prop.~\ref{propexpansionm},
\begin{align*}
\frac 1{x} \int_0^{x} m(s,z)\,ds & = -k - \frac 1x \int_0^x \int_0^1 V(s+t) e^{-2kt} \dd t \dd s \\
& \qquad + \frac 1{kx} \int_0^x \int_0^1 \int_0^{t_1} e^{-2kt_1} (1-e^{-2kt_2}) V(s+t_1) V(s+t_2) \dd t_2 \dd t_1 \dd s + O(\lvert k\rvert^{-2}) 
\end{align*}
with the error term independent of $x$. For the term linear in $V$, we use $p=s+t$ to rewrite the iterated integral as $ \int_0^1 \int_t^{x+t} V(p) e^{-2kt} dp dt$. Then we wish to note that
\begin{equation}\label{30nov1}
\frac 1x \int_0^1 \int_t^{x+t} V(p) e^{-2kt} \dd p \dd t =  \frac 1x \int_0^1 \int_0^{x} V(p) e^{-2kt} \dd p \dd t + O(x^{-1}), \qquad x\to\infty,
\end{equation}
for any $k$. This is because the two iterated integrals describe similar regions in $\bbR^2$: the symmetric difference of the regions $\{(t,p) \mid 0 \le t \le 1, t \le p \le x+t\}$ and $\{(t,p) \mid 0 \le t \le 1, 0 \le p \le x\}$ is contained in $[0,1]\times ([0,1] \cup [x,x+1])$, and the double integral over that region is bounded uniformly in $x$ due to \eqref{L1locunif}. Now the integral in \eqref{30nov1} separates and simplifies using $\int_0^1 e^{-2kt} \dd t = \frac 1{2k} + O(e^{-2\Re k})$.
By analogous arguments, using $q=s+t_2$ to rewrite the quadratic term and comparing the regions
$\{ (t_1,t_2,q) \mid 0 \le t_2 \le t_1 \le 1, t_2 \le q \le x+t_2 \}$ and $\{ (t_1,t_2,q) \mid 0 \le t_2 \le t_1 \le 1, 0 \le q \le x \}$, 
\begin{align*}
&  \frac 1{kx} \int_0^x \int_0^1 \int_0^{t_1} e^{-2kt_1} (1-e^{-2kt_2}) V(s+t_1) V(s+t_2) \dd t_2 \dd t_1 \dd s \\
& = \frac 1{kx} \int_0^1 \int_0^{t_1} \int_0^x  e^{-2kt_1} (1-e^{-2kt_2}) V(q+t_1-t_2) V(q) \dd q \dd t_2 \dd t_1 + O(x^{-1}) \\
& = \frac 1{kx} \int_0^1 \int_0^x  h(u) V(q+u) V(q) \dd q \dd u + O(x^{-1}) 
\end{align*}
as $x \to \infty$, for any $k$. For the last step we introduced $u = t_1 - t_2 \in [0,1]$ and $h(u) = \int_0^{1-u} e^{-2k(u+t_2)} (1-e^{-2kt_2}) \dd t_2$. For the remaining double integral, it is elementary to estimate that $h(u) = O(\lvert k\rvert^{-1})$ uniformly in $u \in [0,1]$ and that 
\[
\frac 1x \int_0^1 \int_0^x  \lvert V(q+u) V(q) \rvert \dd q \dd u \le C^2
\]
where $C$ denotes the $\sup$ in \eqref{L1locunif}, so \eqref{30nov3} follows.
\end{proof}


\section{Regular measures for half-line Schr\"odinger operators}
The main part of this section is devoted to the study of limits of the function
\begin{align}\label{eq:defineh}
h(x,z):=\frac{1}{x}\log|u(x,z)|,
\end{align}
as $x\to\infty$. Our first goal is to show that for $z\in\C_+$ we have that $\liminf_{x\to\infty}h(x,z)\geq 0$. 

\begin{lemma}\label{lem:positiveLimit}
Fix $\z\in\C_+$. Then
\begin{align*}
\liminf_{x\to\infty}\frac{1}{x}\log|u(x,\z)|\geq 0.
\end{align*}
\end{lemma}
\begin{proof}
Note first of all that $u(x,\z)\neq 0$ whenever $x > 0$, because the converse would correspond to an complex eigenvalue for the self-adjoint realization of $L_V$ on $[0,x]$ with Dirichlet boundary conditions. The Weyl solution $\psi(x,z)$ is an eigensolution and is in $L^2((0,\infty))$; the condition \eqref{L1locunif} is sufficient to conclude that $\psi$ decays pointwise \cite[Theorem 1.1]{Lu13}, i.e.
\begin{align*}
\lim\limits_{x\to\infty}\psi(x,\z)=0.
\end{align*}
Combining with Lemma~\ref{lemmaDirichletToWeyl} shows that $\lvert u(x,z) \rvert \to \infty$ as $x\to\infty$, which completes the proof.
\end{proof}
Let $\E=\sigma_\ess(L_V)$ written in the form \eqref{eq:spectrumGaps}. That is $b_0=\min \E$ and $(a_j,b_j)$ denote the gaps of $\E$.
\begin{lemma}\label{lem:OneEigenvalue}
For any $\e>0$ there exists $x_0>0$ such that $u(x,z)\neq 0$ for $x>x_0$ and $z\leq b_0-\e$. Moreover, let $n_j(\e)$ denote the finite number of eigenvalues in $(a_j+\e,b_j-\e)$. Then, for any $x>0$, $u(x,z)$ has at most $n_j(\e)+1$ zeros in $(a_j+\e,b_j-\e)$.
\end{lemma}
\begin{proof}
Since $L_V$ is semibounded there are at most finitely many eigenvalues below $b_0-\e$. Hence, the first statement follows by Sturm oscillation theory. 

As in the proof of Lemma \ref{lemmaDirichletToWeyl}, we use the spectral theoretic Green's function $g(x,x;z)$. By the Weyl $M$-matrix representation for $L_V$ centered at $x$, $g(x,x;\cdot)$ is analytic on $\C\setminus\sigma(L_V)$ and, since it is Herglotz, it is strictly increasing on intervals in $\R \setminus \sigma(L_V)$. In particular, every pole of $g(x,x;\cdot)$ is an eigenvalue of $L_V$, so it has at most $n_j(\e)$ poles in $(a_j+\e,b_j-\e)$. By \eqref{23dec1}, every zero of $u(x,z)$ is a pole of $-(\partial_x u)(x,z) / u(x,z)$ and a zero of $g(x,x;\cdot)$. Since zeros and poles of the Herglotz function $g(x,x;\cdot)$ strictly interlace on intervals in the domain of meromorphicity, it follows that $u(x,z)$ has at most $n_j(\e)+1$ zeros in $(a_j+\e,b_j-\e)$.
\end{proof}

We are now ready to study the existence of limit points for the family of functions $\cF=\{h(x,z)\}_{x\in [1,\infty)}$. Since $u(x,\cdot)$ are entire functions, the functions $h(x,\cdot)$ are subharmonic in $\C$, and they can be viewed as elements of the space of distributions $\cD'(\C)$ with nonnegative distributional Laplacian.

\begin{theorem}\label{thm:harmonicLimitExtension}
\begin{enumerate}[(a)]
\item The family $\cF=\{h(x,z)\}_{x\in [1,\infty)}$ is precompact in $\cD'(\C)$.
\item For any sequence $(x_j)_{j=1}^\infty$ with $x_j \to \infty$ such that $h(x_j,\cdot)$ converges in $\cD'(\C)$, the limit  $h = \lim_{j\to\infty} h(x_j,\cdot)$ is also a subharmonic function on $\C$, harmonic on $\C \setminus \E$, and $h(x_j,\cdot)$ also converge to $h$ uniformly on compact subsets of $\C \setminus \E$.
\end{enumerate}
\end{theorem}

\begin{proof}
(a) By Corollary \ref{cor:uniformUpperBound}, $h(x,z)$ is uniformly bounded from above on compact subsets of $\C$. Moreover, Lemma \ref{lem:positiveLimit} implies a pointwise lower bound at some arbitrary point $z_0\in\C_+$. Hence, \cite[Theorem 4.1.9]{Hoermander1}
shows that $\cF$ is precompact in the topology of $\cD'(\C)$. 

(b) On $\bbC_+$ and on $\bbC_-$, the functions $h(x,z)$ are harmonic and uniformly bounded above. Since they are also pointwise bounded below, they are uniformly bounded and uniformly equicontinuous on each compact subset of $\bbC_\pm$. Therefore, they are precompact in the topology of uniform convergence on compact subsets of $\bbC_\pm$. Since this convergence implies convergence in $L^1_\loc(\C_\pm)$, it follows that if the sequence $h(x_j,\cdot)$ converges in $\cD'(\C)$ to $h$, then it also converges to $h$ uniformly on compact subsets of $\C_\pm$.

Next, we show that $h$ has a harmonic extension through an arbitrary gap $(a_m, b_m)$ of $\E$. Fix $\e>0$. By Lemma \ref{lem:OneEigenvalue}, there are at most $n_m(\e)+1$ zeros of $u(x_j,\z)$ in $(a_m+\e,b_m-\e)$. Let $p_j$ be the monic polynomial of degree at most $n_m(\e)+1$ which vanishes exactly at these zeros. Now consider 
$$
f_j(\z)=\frac{1}{x_j}\log\left|\frac{u(x_j,\z)}{p_j(z)}\right|,
$$
which is harmonic on $\C_+ \cup \C_- \cup (a_m+\e,b_m-\e)$.On the boundary of the rectangle $(a_m-1,b_m+1) \times (-1,1)$, $p_j$ is uniformly bounded below by $1$, so by the maximum principle, the analytic functions $\frac{u(x_j,\z)}{p_j(z)}$ are also bounded above by $e^{cx_j}$ 
in this rectangle for some constant $c$. Hence, $f_j(z)$ is locally uniformly bounded above on $R_m = (a_m+\epsilon,b_m-\epsilon) \times (-1,1)$. Since all zeros of $p_j$ are in $(a_m,b_m)$, there is still a pointwise lower bound for $z_0\in\C_+$. Hence, the functions $f_j$ are harmonic on $R_m$ and precompact in the topology of uniform convergence on compacts. For any $z \in R_m \setminus \bbR$,  
\[
\lim_{j \to \infty} (h_j(z) - f_j(z))  = \lim_{j\to \infty} \frac 1{x_j} \log \lvert p_j(z) \rvert = 0
\]
since $ \lvert \Im z \rvert^{n_m(\e)+1} \le \lvert p_j(z) \rvert \le (b_m - a_m+1)^{n_m(\e)+1}$. Hence, any subsequential limit of the $f_j(z)$ is a harmonic function on $R_m$ which agrees with $h$ on $R_m \setminus \R$. It follows that $f_j$ converge in $R_m$ uniformly on compacts, so it provides a harmonic extension for $h$ through $(a_m+\e,b_m-\e)$. Since $\e>0$ was arbitrary and the extensions must coincide on their common domain, we obtain an extension through $(a_m,b_m)$ by letting $\e\to 0$. It follows from the weak identity principle for subharmonic functions \cite[Theorem 2.7.5]{RansPotential} that the harmonic extension coincided with $h$.

Consider a compact $K \subset \bbC \setminus [b_0,\infty)$. By possibly increasing $K$, assume that $K \not\subset \R$. Choose an open set $U$ such that $K \subset U \subset \overline{U} \subset \bbC \setminus [b_0,\infty)$. By Lemma \ref{lem:OneEigenvalue}, for all sufficiently large $j$, $h_j(z)$ is harmonic in $U$. The functions $h_j$ are uniformly bounded above and pointwise bounded below at $z_0 \in K \cap (\C_+\cup \C_-)$, so they form a precompact sequence with respect to uniform convergence on $K$. As before, every limit is equal to $h$, so $h_j$ converge to $h$ uniformly on compacts.
\end{proof}

Collecting our results now yields that the limits define a positive harmonic function in $\Omega=\C\setminus \E$.
\begin{theorem}\label{thm:MainTheoremLimit}
Let $x_j \to \infty$ be a sequence such that $h_j = h(x_j,\cdot)$ converge in $\cD'(\C)$. Then  $h=\lim\limits_{j\to\infty}h_j$ defines a positive harmonic function in $\Omega$, the limit
\begin{equation}\label{asubseqlimit}
a=\lim\limits_{j\to\infty}\frac{1}{x_j}\int_{0}^{x_j}V(x)\dd x
\end{equation} 
exists, and $h$ has the nontangential asymptotic behavior
\begin{equation}\label{12dec1}
h(z)=\Re\left(k+\frac{a}{2k}\right)+O(\lvert k\rvert^{-2}),
\end{equation}
$z\to\infty$, $\delta \le \arg z \le 2 \pi - \delta$, for any $\delta > 0$.
\end{theorem}
\begin{proof}
Harmonicity of $h$ was proved in Theorem  \ref{thm:harmonicLimitExtension} and positivity in $\C_+\cup \C_-$ follows from Lemma \ref{lem:positiveLimit}. That $h$ is also positive in $\R\setminus \E$ follows by the maximum principle for harmonic functions, and by Corollary~\ref{corollaryDirichletToWeyl},
\begin{equation}\label{13dec1}
h(z) = - \lim_{j\to\infty} \frac 1{x_j} \Re \int_0^{x_j} m(x,z)\dd x.
\end{equation}
Denote $c = \min \sigma(L_V)$. By general spectral theory, then $m(x,z)$ are analytic functions on $\C\setminus [c,\infty)$ and $m(x,z) < 0$ on $(-\infty,c)$. Since convergence of analytic functions follows from convergence of their real parts together with convergence at one point, from $\Im m(x,z) = 0$ for $z < c$ together with \eqref{13dec1}, it follows that the limit
\[
w(z) = \lim_{j\to\infty} \frac 1{x_j} \int_0^{x_j} m(x,z) \dd x
\]
converges uniformly on compact subsets of $\C \setminus [c,\infty)$. If $a$ denotes some accumulation point of the sequence $\frac{1}{x_j}\int_{0}^{x_j}V(x)\dd x$, applying Corollary~\ref{cor:maveragetoVaverage} along the subsequence and using uniformity of the error term, it follows that
\begin{equation}\label{16dec1}
w(z) = - k - \frac{a}{2k} + O(\lvert k\rvert^{-2})
\end{equation}
nontangentially as $z \to \infty$, with $\arg z \in [\delta,\pi-\delta]$. This asymptotic behavior can only hold for one value of $a$, so it follows that the limit \eqref{asubseqlimit} exists. 
 or, by symmetry, $\arg z \in [\pi+\delta,2\pi-\delta]$.

We know that \eqref{16dec1} holds as $z \to \infty$  with $\arg z \in [\delta,\pi-\delta]$ and, by symmetry, for $\arg z\in [\pi+\delta,2\pi-\delta]$. It remains to extend this asymptotic behavior to a sector of the form $\arg z \in [\pi-\delta,\pi+\delta]$. Without loss of generality assume $c=0$. Since $\Re w = - h \le 0$, the function $f(\lambda) = - i w(\lambda^2)$ is Herglotz, and obeys
\begin{align}\label{eq:proofPhLind}
f(\lambda) = \lambda - \frac{a}{2\lambda} + O(|\lambda|^{-2}), \quad \lvert \lambda \rvert \to \infty,
\end{align}
along the rays $\arg \lambda = \pi/2-\delta/2$ and  $\arg \lambda = \pi/2+\delta/2$. In the sector $T =\{\lambda: \pi/2-\delta/2\leq \arg \lambda\leq \pi/2+\delta/2\}$, the function $g(\lambda) = \lambda^2\left(f(\lambda) -\lambda + \frac{a}{2\lambda} \right)$ is analytic. It has a continuous extension to $\overline{T}$ with $g(0) = 0$, because $f(\lambda) = O(1/\lambda)$ as $\lambda \to 0$ nontangentially. By \eqref{eq:proofPhLind}, $g$ is bounded on the boundary of $T$. Finally, since $f$ is Herglotz, $f,g$ grow at most polynomially as $\lambda \to \infty$, $\lambda \in T$, so by Phragm\'en--Lindel\"of, $g$ is bounded in $T$. This implies that $f$ has the asymptotic behavior \eqref{eq:proofPhLind} also in the sector $T$. Rewriting the conclusion for $w$ and $h = -\Re w$ completes the proof.
\end{proof}

We need the following variant of the Herglotz representation:

\begin{lemma}\label{lem:Herglotz}
Let $f$ be a Herglotz function. Assume that $\Im f(iy)=O(y^{-1})$ as $y\to\infty$. Then for some $\b\in\R$
\begin{align*}
f(\l)=\b+\int_\R\frac{\dd\mu(t)}{t-\l},\quad\text{with }\lim\limits_{y\to\infty}y\Im f(iy)=\mu(\R)<\infty,
\end{align*}
and
\begin{align}\label{eq:BorelTransform}
f(\l)=\b-\frac{\mu(\R)}{\l}+o(|\l|^{-1}),
\end{align}
$\l\to\infty$, $\delta \le \arg \l \le  \pi - \delta$, for any $\delta > 0$.
\end{lemma}
\begin{proof}
Starting from the Herglotz representation, we can write 
$
\Im f(iy)=ay+\int\frac{y}{t^2+y^2}\dd\mu(t),
$
with 
$
\lim\limits_{y\to\infty}\frac{\Im f(iy)}{y}=a.
$
Hence, by our assumption, $a=0$. Moreover, by monotone convergence
\begin{align*}
\lim\limits_{y\to\infty}y \Im f(iy)=\lim_{y\to\infty}\int\frac{y^2}{t^2+y^2}\dd\mu(t)=\mu(\R).
\end{align*}
By our assumption, this shows that $\mu(\R)<\infty$. We have $\l\int_\R\frac{\dd\mu(t)}{t-\l}+\mu(\R)=\int_\R\frac{t}{t-\l}\dd\mu(t)\to 0$ as $\l\to\infty$, by dominated convergence since $\left|\frac{t}{t-\l}\right|\leq \frac{1}{\sin\d}$.
\end{proof}

We are now ready to prove an asymptotic expansion \eqref{Martinexp2term} of higher order for $M_\E$. 

\begin{proof}[Proof of Theorem \ref{thm:MainAkhiezerLevin}]
By translation, we may assume that $0=\min \E$. By precompactness  of the family $\{h(x,z)\}_{x\geq 1}$, there is a sequence $x_n\to\infty$ for which the limit $h = \lim_{n\to\infty} \frac 1{x_{n}} \log \lvert u(x_{n},\cdot) \rvert$ is convergent in $\cD'(\C)$. By Theorem~\ref{thm:MainTheoremLimit}, $h$ is a positive harmonic function in $\Omega$ and $h(z) / \sqrt{-z} \to 1$ as $z \to -\infty$, so by Lemma~\ref{lem:AkhiezerLevin}, $\Omega$ is Greenian, obeys the Akhiezer--Levin condition, and $h \ge M_\E$ in $\Omega$.   Using \eqref{eq:MartinLowerBound}, we obtain for  $z\in\Omega$ 
\begin{align}\label{eq:proofAsympM1}
\Re\sqrt{-z} \leq M_\E(z)\leq h(z).
\end{align}
Hence, the difference $M_\E(-k^2)-\Re k$ defines a positive harmonic function in $\Omega$ and \eqref{12dec1}, \eqref{eq:proofAsympM1} imply that $M_\E(-k^2)-\Re k= O(|k|^{-1})$. 
Set $z=\l^2$ and $v(\l)=M_\E(-k^2)-\Re k$. We thus obtain a positive harmonic function in $\C_+$ such that
$v(iy)= O(y^{-1})$. By Lemma \ref{lem:Herglotz} there is a constant $c$ such that 
\begin{align*}
v(\l)=-\Im\left(\frac{c}{\l}\right)+o(|\l|^{-1})
\end{align*}
as $\lambda \to \infty$ nontangentially in $\C_+$. Recalling that $\l=ik$, this shows that 
\begin{align*}
M_\E(-k^2)-\Re k=\Re\left(\frac{c}{k}\right)+o(|k|^{-1}).
\end{align*}
This completes the proof. 
\end{proof}

\begin{proof}[Proof of Theorem \ref{thm11}]
Consider a sequence $x_n \to \infty$ such that
\[
\lim_{n\to\infty} \frac 1{x_n} \int_0^{x_n} V(t)\,dt = \liminf_{x\to\infty} \frac 1{x} \int_0^{x} V(t)\,dt.
\]
Due to Theorem \ref{thm:harmonicLimitExtension}, this sequence has a subsequence for which the limit
$
h = \lim_{j\to\infty} \frac 1{x_{n_j}} \log \lvert u(x_{n_j},\cdot) \rvert
$
is convergent in $\cD'(\C)$. As in the proof of Theorem  \ref{thm:MainAkhiezerLevin}, we have $h \ge M_\E$ in $\Omega$.  Theorem \ref{thm:MainAkhiezerLevin} and Theorem \ref{thm:MainTheoremLimit} yield 
\[
a_\E= \lim_{k\to + \infty} 2k (M_\E(-k^2) - k)\le \lim_{k\to + \infty} 2k (h(-k^2) - k) = \lim_{j\to\infty} \frac 1{x_{n_j}} \int_0^{x_{n_j}} V(s)\,ds. \qedhere
\]
\end{proof}

\begin{proof}[Proof of Theorem \ref{thm13}]
Fix $z_0 \in \bbC \setminus [\min\E,\infty)$ and consider a sequence $x_n \to \infty$ such that
\[
\lim_{n\to\infty} \frac 1{x_n} \log \lvert u(x_n, z_0) \rvert = \liminf_{x\to\infty} \frac 1{x} \log \lvert u(x,z_0) \rvert.
\]
We can again pass to a subsequence such that $h = \lim_{j\to\infty} \frac 1{x_{n_j}} \log \lvert u(x_{n_j},\cdot) \rvert$ and $h \ge M_\E$ in $\Omega$. In particular,
\[
 \liminf_{x\to\infty} \frac 1{x} \log \lvert u(x,z_0) \rvert = h(z_0) \ge M_\E(z_0). \qedhere
  \]
\end{proof}

\begin{proof}[Proof of Theorem \ref{thm14}]
By inclusions, we have $(vi)\implies (iv)$ and $(v)\implies (iv)$. 

$(iv)\implies (vi)$: Consider any sequence $x_j \to \infty$ such that $h = \lim_{j\to\infty} h(x_j,\cdot)$ converges. On $\C_+$, the limit $h$ obeys $h \le M_\E$ by (iv) and $h \ge M_\E$ by Theorem~\ref{thm13}. Thus, $h=M_\E$ on $\C_+$, and then on $\Omega$ by harmonic continuation. Thus, $M_\E$ is the only possible subsequential limit of $h(x,\cdot)$ as $x\to \infty$, so by precompactness, $\lim_{x\to\infty} h(x,z) = M_\E(z)$ uniformly on compact subsets of $\C\setminus [b_0,\infty)$.

$(vi)\implies (v)$: Given $(vi)$, we know that for any convergent sequence $h(x_n,z)$ the limit is $M_\E$. For $z\in[b_0,\infty)$ we have by \cite[Theorem 2.7.4.1]{Azarin09} that
\begin{align*}
\limsup_{n\to\infty}h(x_n,z)\leq (\limsup_{n\to\infty}h(x_n,z))\,\check{\vrule height1.3ex width0pt}=M_\E(z),
\end{align*}
where $\check f$ denotes the upper semicontinuous regularization of $f$. The first inequality follows by the general fact that $f\leq  \check f$. 

$(v)\implies (ii)$: This follows from Theorem \ref{lem:ConePositive}.

$(ii)\implies(iii)$: Due to \cite[Corollary 6.4]{GarHarmonicMeasure} the set of Dirichlet-irregular points is polar and thus, by \cite[Theorem 8.2]{GarHarmonicMeasure} it is of harmonic measure zero  and the claim follows.

$(iii)\implies(vi)$: Take a sequence $x_n\to\infty$ such that $\lim_{n\to\infty} h(x_n,z)=h(z)$ in $\cD'(\C)$ and uniformly on compact subsets of $\C\setminus [b_0,\infty)$. Due to \cite[Theorem 2.7.4.1]{Azarin09}, there is a polar set $X_1$ such that for any $z\in \C\setminus X_1$,
\begin{align*}
\limsup\limits_{n\to\infty}h(x_n,z)=h(z).
\end{align*}
On the other hand, assuming $(iii)$, there exists $X_2$ with $\omega_\E(X_2,z_0)=0$, such that for $t\in\E\setminus (X_1\cup X_2)$ by upper semicontinuity
\begin{align*}
0\leq \liminf\limits_{\substack{z\to t\\ z\in\Omega}} h(z)\leq\limsup\limits_{\substack{z\to t\\ z\in\Omega}} h(z)\leq h(t)\leq 0.
\end{align*}
Since $\omega_\E(X_1\cup X_2,z_0)=0$, Theorem \ref{lem:ConePositive} gives $h=cM_\E$. Comparing the leading order asymptotic behavior at $\infty$ shows that $c=1$. Thus, $M_\E$ is the only possible subsequential limit of $h(x,\cdot)$ as $x\to \infty$, so by precompactness, $\lim_{x\to\infty} h(x,z) = M_\E(z)$ uniformly on compact subsets of $\C\setminus [b_0,\infty)$.

$(vi)\implies(i)$: By Theorem \ref{thm:MainTheoremLimit}, $(vi)$ implies that $\frac 1{x_j} \int_0^{x_j} V(t)\dd t \to a_\E$ for every sequence $x_j \to \infty$, so $(i)$ follows.

$(i)\implies (vi)$: Take a sequence $x_n\to\infty$ such that $h=\lim_{n\to\infty} h(x_n,\cdot)$ converges in $\cD'(\C)$. Define $v(\l)=h(-k^2)-M(-k^2)$. Similarly to the proof of Theorem \ref{thm:MainAkhiezerLevin}, this yields a positive harmonic function in $\C_+$. By Theorem \ref{thm:MainTheoremLimit} and Theorem \ref{thm:MainAkhiezerLevin}, $v(iy)= o\left(y^{-1}\right)$ as $y\to\infty$. By Lemma \ref{lem:Herglotz}, $\lim_{y\to\infty} y v(iy)=0$ implies that $v\equiv 0$. This shows that $M_\E$ is the only subsequential limit of $h(x,\cdot)$ as $x\to\infty$. By precompactness, $(vi)$ follows.
\end{proof}

The functions $u(x,z)$ are entire functions of order $\frac{1}{2}$ and as such admit a product representation
\begin{align*}
u(x,z)=u(x,z_*)\prod_{j=1}^{\infty}\left(1-\frac{\z-\z_*}{\z_j-\z_*}\right),
\end{align*}
where $z_j$ depend on $x$ and $\z_*$ is some normalization point. Then the Riesz measure, $\rho_x$, of the subharmonic function $\log|u(x,z)|$ is a rescaled zero counting measure of $u(x,z)$. That is,
\begin{align*}
\frac 1x \log|u(x,z)|=\frac 1x \log|u(x,z_*)|+\int\log\left|1-\frac{\z-\z_*}{t-\z_*}\right|\dd\rho_x(t),
\end{align*}
where $\rho_x$ is defined in \eqref{eqn:rhoxdefn}.

\begin{proof}[Proof of Theorem~\ref{thm:dos1}] 
By Theorem \ref{thm14} and Theorem  \ref{thm:harmonicLimitExtension}, $h(x,\cdot)\to M_\E$ in $\cD'(\C)$ as $x\to\infty$. By the definition of the Riesz measure, for any $\phi\in C^\infty_c(\C)$,
\begin{align*}
\lim_{x\to\infty}2\pi\int \phi(z)\dd\rho_{x}(z)=\lim_{n\to\infty}\int h(x,z)\Delta \phi(z)\dd \l(z)=\int M_\E(z)\Delta \phi(z)\dd \l(z)=2\pi\int \phi(z)\dd\rho_\E(z),
\end{align*}
where $\dd\l$ denotes the Lebesgue measure on $\C$. The rest follows from density of $C^\infty_c(\C)$ in $C_c(\C)$.
\end{proof}

\begin{proposition}\label{prop:MeasureConverg} 
Let $\dd \mu$ be the spectral measure of $L_V$, where $V$ satisfies \eqref{L1locunif} and $\sigma_\ess(\dd\mu)=\E$. Suppose that along a sequence $x_n\to\infty$ the Riesz measure $\dd\rho_{x_n}$ converge to $\rho_\E$ in the weak-$*$ sense. Then, either $h(x_n,z)$ converge to $M_\E(z)$ or there exists a polar Borel set $X$ such that $\mu(\R\setminus X)=0$.
\end{proposition}
\begin{proof}
Assume that $h(x_n,\cdot)$ do not converge to $M_\E$ and consider a subsequence $x_{n_j}$ such that $h(x_{n_j},\cdot)\to h$ in $\cD'(\C)$ with some limit $h$ not equal to $M_\E$.
By the upper envelope theorem \cite[Theorem 2.7.4.1]{Azarin09} there is a polar set $X_1$ such that for any $z\in \C\setminus X_1$,
\begin{align*}
\limsup\limits_{j\to\infty}h(x_{n_j},z)=h(z).
\end{align*}
The subharmonic function $h$ has some Riesz measure $\rho$ and by the same arguments as in the proof of Theorem~\ref{thm:dos1}, $\rho_{x_{n_j}}$ converges to $\rho$ in the weak-$*$ sense. Hence, by uniqueness of the limits our assumption implies that $\rho=\rho_\E$ and, by Lemma~\ref{lem:Levin1} applied to $h$ and $M_\E$,
\begin{align*}
h(z)=h(z_*)+ \int\log\left|1-\frac{z-z_*}{t-z_*}\right|\dd\rho_\E(t) = d + M_\E(z)
\end{align*}
where $d = h(z_*) - M_\E(z_*)$. Recall that $M_\E$ has a unique subharmonic extension to $\C$ which vanishes q.e. on $\E$. Therefore, there is a polar set $X_2$ such that $h(z)=d$ for $z\in\E\setminus X_2$. Moreover, since $M_\E\leq h$ on $\Omega$ we see that $d \ge 0$, and since $h$ is not equal to $M_\E$, $d > 0$. In particular,
\begin{align*}
\limsup\limits_{j\to\infty}h(x_{n_j},z)=d>0, \qquad \forall z\in\E\setminus (X_1\cup X_2).
\end{align*} 
However, by Schnol's theorem \cite{Schnol}, for $\mu$-a.e. $z\in \E$, the Dirichlet solution decays at most polynomially and, in particular,
\[
\limsup\limits_{j\to\infty}h(x_{n_j},z)\leq 0.
\]
Thus $\mu(\E\setminus (X_1\cup X_2)) = 0$, which implies the claim with $X = X_1 \cup X_2$.
\end{proof}

In particular, Theorem~\ref{thm:dos2} is now proved.

\begin{proof}[Proof of Theorem \ref{thm:Widom}]
	By Schnol's theorem \cite{Schnol} for $\mu$-a.e. $z\in\E$ 
	\begin{align}\label{eq:dec8}
	\limsup\limits_{x\to\infty}h(x,z)\leq 0.
	\end{align}
	Hence, by assumption, \eqref{eq:dec8} holds $\omega_\Omega(\cdot,z_0)$-a.e.. Therefore, $V$ is regular by Theorem \ref{thm14}.
	\end{proof}
	
\section{Applications}\label{sec:Applications}

\begin{proof}[Proof of Theorem~\ref{corL1Cesaro}] 
(a) Denoting $\E = \sigma_\ess(L_V)$, it follows from $\E \subset [0,\infty)$ that $M_\E$ is a positive harmonic function on $\bbC \setminus [0,\infty)$. Since the Martin function for the domain $\bbC \setminus [0,\infty)$ is $\Re \sqrt{-z}$, it follows from Lemma~\ref{lem:AkhiezerLevin} that $M_\E(z) \ge \Re \sqrt{-z}$. Comparing this with the asymptotic expansion \eqref{Martinexp2term} as $z \to -\infty$ shows that $a_\E \ge 0$ so, by \eqref{aEinequality1}, $\liminf_{x\to \infty}\frac{1}{x}\int_0^xV(t)\dd t\geq 0$. 

(b) As in (a), $a_\E \ge 0$. By \eqref{aEinequality1} and $\liminf_{x\to \infty}\frac{1}{x}\int_0^xV(t)\dd t\leq 0$, this implies that $a_\E=0$. Moreover, $M_\E(z) - \Re \sqrt{-z}=o(\sqrt{|z|}^{-1})$ defines a positive harmonic function in $\C\setminus[0,\infty)$ so, by Lemma \ref{lem:AkhiezerLevin}, $M_\E(z)=\Re \sqrt{-z}$ . If $\E$ was a proper subset of $[0,\infty)$, since $\E$ is closed, there would exist a gap $(a,b) \subset [0,\infty) \setminus \E$, and on this gap $M_\E$ would be strictly positive, contradicting $M_\E(z) = \Re \sqrt{-z}$.

(c) Again by $a_\E \ge 0$ and \eqref{aEinequality1}, $\limsup_{x\to \infty}\frac{1}{x}\int_0^xV(t)\dd t\leq 0$ implies that $V$ is regular. 
\end{proof}

We now turn to the construction of a potential which is regular for $\E = [0,\infty)$ but not decaying, even in the Ces\`aro sense. The potential will be constructed piecewise, so we begin by considering a $2\delta$-periodic potential defined by
\[
W_\delta(x) = \begin{cases}
1 & x \in [0,\delta) \\
-1 & x \in [\delta,2\delta)
\end{cases}
\]
Let us compute the discriminant $\Delta_\delta(z)$ and the smallest eigenvalue for the periodic problem,
\[
\lambda_\delta = \min \{ \lambda \in \bbR \mid \Delta_\delta(\lambda) = 2 \}.
\]

\begin{lemma} $\lim_{\delta \downarrow 0} \lambda_\delta = 0$.
\end{lemma}

\begin{proof}
Since $\lvert W_\delta \rvert \le 1$ and $\lambda_\delta$ is the minimum of the periodic spectrum, by standard variational principles, $\lambda_\delta \in [-1,1]$ for all $\delta > 0$. 
The transfer matrix corresponding to $W_\delta$ at energy $\lambda \in (-1,1)$ is
\[
T_\delta(\lambda) = \begin{pmatrix}
\cosh(\delta \sqrt{1-\lambda}) & \frac{ \sinh(\delta \sqrt{1-\lambda}) }{ \sqrt{1-\lambda} } \\
\sqrt{1-\lambda} \sinh(\delta \sqrt{1-\lambda}) & \cosh(\delta \sqrt{1-\lambda})
\end{pmatrix}
\begin{pmatrix}
\cos(\delta \sqrt{1+\lambda}) & \frac{ \sin(\delta \sqrt{1+\lambda}) }{ \sqrt{1+\lambda} } \\
-\sqrt{1+\lambda} \sin(\delta \sqrt{1+\lambda}) & \cos(\delta \sqrt{1+\lambda})
\end{pmatrix}
\]
From this it is elementary to obtain the asymptotic behavior for the discriminant, $\Delta_\delta(\lambda) = \tr T_\delta(\lambda)$, in the form
\begin{equation}\label{discriminantasymptotics}
\Delta_\delta(\lambda) = 2 - 4 \lambda \delta^2 + O(\delta^3), \qquad \delta \downarrow 0,
\end{equation}
uniformly in $\lambda \in (-1,0)$ (and then, by continuity, for $\lambda \in [-1,0]$). From this, it follows that for any $t < 0$, there exists $\delta_0 > 0$ such that $\delta \in (0,\delta_0)$ and $\lambda \in [-1,t)$ implies $\Delta_\delta(\lambda) > 2$ and therefore $\lambda_\delta \ge t$. It follows that $\liminf_{\delta\downarrow 0} \lambda_\delta \ge 0$.

Meanwhile, $\Delta_\delta(0)=2 \cosh \delta \cos \delta = 2 - \delta^4 / 3 + o(\delta^4)$ as $\delta \to 0$ implies that $\limsup_{\delta\downarrow 0} \lambda_\delta \le 0$.
\end{proof}

\begin{proof}[Proof of Example~\ref{thmnotCesaro}]
Consider the Dirichlet solution $u(x,t)$ corresponding to the given potential at some $t < 0$. There exists $n_0$ such that for all $n \ge n_0$, $\lambda_{1/(2n)} > t$. At energies below the periodic spectrum, transfer matrices have strictly positive entries; applying this on intervals $[n,n+1]$ and since products of matrices with positive entries have positive entries, we conclude that $u(x,t)$ has at most one zero with $x > n_0-1$. Since zeros of an eigensolution are isolated, it follows that $u(\cdot,t)$ has finitely many zeros, so by Sturm oscillation theory, $\min \sigma_\ess(L_V) \ge t$. Since this holds for arbitrary $t < 0$, we conclude $\min \sigma_\ess(L_V) \ge 0$.

Conversely, since $V$ obeys $\lim_{x\to \infty} \frac 1x \int_0^x V(t) \dd t = 0$, the statement is completed by Theorem~\ref{corL1Cesaro}.
\end{proof}

\begin{proof}[Proof of Example \ref{ex:Sparse1}]
For $x \in [x_n, x_{n+1}]$ we have 
$
\frac{1}{x}\int_{0}^{x}V(t)\dd t\leq \int W(t)\dd t\frac{n+1}{x_n}.
$
Since the condition on $x_n$ implies that $\frac{x_n}{n}\to\infty$ we see that $\lim_{x\to\infty}\frac{1}{x}\int_{0}^{x}V(t)\dd t=0$. Since $V \ge 0$, we have $\sigma_\ess(L_V) \subset \sigma(L_V) \subset [0,\infty)$, so by Theorem~\ref{corL1Cesaro}, $V$ is regular and $\sigma_\ess(L_V) = [0,\infty)$.

Let $H_W$ be the whole-line operator with the potential $W(x)$. Since $f\geq 0$, we have $\sigma(H_W)\subset[0,\infty)$. Hence, we conclude that $\min \sigma(H_{-W})< 0$, for otherwise \cite[Corollary 1]{DamanikKillipSimonCMP} would imply that $W\equiv 0$. Now by \cite[Theorem 7.1]{LastSimon06} it follows that $\sigma_{\ess}(H_{-V})=\sigma(H_{-W})$ (where $H_{-V}$ is the full line operator with potential $V$ extended to $\R_-$ by $V\equiv 0$). Since $\sigma_{\ess}(H_{-V})=\sigma_{\ess}(L_0)\cup \sigma_{\ess}(L_{-V})$ this shows that $\min\sigma_\ess(L_{-V}) < 0$.
\end{proof}

\begin{proof}[Proof of Theorem~\ref{thm:ultimatePasturIshii}]
The Lyapunov exponent $\g$ is harmonic in $\C_+\cup \C_-$ and subharmonic in $\C$. By \eqref{eq:lyaponuv} for a.e. $\eta\in S$ 
\begin{align*}
	\lim\limits_{x\to\infty}\frac 1x \log \lvert u_\eta(x,z)|=\gamma(z)
\end{align*}
converges pointwise in $\C_+\cup\C_-$; by the weak identity principle for subharmonic functions and precompactness, convergence to $\gamma$ is also in $\cD'(\C)$. By Schnol's theorem, for $\mu_\eta$-a.e. $z$,
\begin{align}\label{eq:Dez20}
	\limsup \limits_{x\to\infty}\frac 1x \log \lvert u_\eta(x,z)|\leq 0.
\end{align}
Fix a sequence $x_n \to \infty$. By the upper envelope theorem \cite[Theorem 2.7.4.1]{Azarin09} there is a polar set $X_\eta$ such that for any $z\in \C\setminus X_\eta$,
\begin{align*}
\limsup\limits_{n\to\infty}\frac 1{x_n} \log \lvert u_\eta(x_n,z)|=\g(z).
\end{align*}
On $Q$, $\gamma>0$. Hence, since \eqref{eq:Dez20} holds for $\mu_\eta$-a.e. $z$, we have $\mu_\eta(Q\setminus X_\eta)=0$. 
\end{proof}
\section{Conformal maps}\label{sec:confMaps}
In view of Corollary \ref{cor:Widomhalfline} and the subsequent discussion, it is of great interest if the harmonic measure of the domain $\C\setminus\E$ is absolutely continuous with respect to the Lebesgue measure $\chi_\E(x)\dd x$. Let $z_0<\min E$ and $G_\E(z,z_0)$ be the Green function of $\C\setminus\E$ with pole at $z_0$ and $\Pi_{z_0}$ the associated comb domain, defined by the upper semicontinuous function $\height$. We say that $\Pi_{\z_0}$ satisfies the sector condition if 
\begin{align*}
S_{z_0}(x)=\sup_{y\in (0,\pi)}\frac{\height_{z_0}(y)}{\lvert x-y\rvert}
\end{align*}
is finite for Lebesgue-a.e. $x\in(0,\pi)$. Then, $\omega_\E(\cdot,z_0)$ is absolutely continuous with respect to the Lebesgue measure if and only if $\Pi_{\z_0}$ satisfies the sector condition. 

The proceeding discussion holds for general semibounded sets $\E$ and does not assume that $\E$ is an Akhiezer-Levin set. Let $M$ be the Martin function with pole at $\infty$, normalized at some internal point $z_*$, $\rho$ its Riesz measure and $\Pi$ and $\Theta$ the corresponding comb and comb mapping. There is a similar characterization  for absolute continuity of $\rho$. Let $\height$ be the upper semicontinuous function defining $\Pi$. Then $\rho$ is absolutely continuous with respect to $\chi_\E(x)\dd x$ if and only if the domain contains a Stolz angle at a.e.\ point at the base of the comb, i.e.
\begin{align}\label{eq:sectorConditionMartin}
S(x)=\limsup_{y\to x}\frac{\height(y)}{\lvert x-y\rvert}
\end{align}
is finite for Lebesgue-a.e. $x\in(0,b)$.

Under various conditions on the set $\E$, it is known that the conformal map $i\Theta'$ has a product representation. We now provide a general proof which does not assume Dirichlet-regularity or any other additional assumptions.
\begin{lemma}\label{lemma:productformula}
	Let $\E$ be a closed non-polar set of the form \eqref{eq:spectrumGaps}. For each $j$ there exists $c_j \in [a_j,b_j]$ such that $M$ is strictly increasing on $(a_j,c_j)$ and strictly decreasing on $(c_j,b_j)$, and $\Theta'(z)$ is given on $z\in \C \setminus [b_0,\infty)$ by
	\begin{equation}\label{eqn:productformula}
	i\Theta'(z) = \frac{C}{\sqrt{b_0-z}} e^{ \int_{[b_0,\infty) \setminus \E} \xi(x) \frac{1+xz}{x-z}\frac{\dd x}{1+x^2}}
	\end{equation}
	where $\xi(x) = 1/2$ for $x \in (a_j,c_j)$, $\xi(x)=-1/2$ for $x \in (c_j,b_j)$, $\xi(x)=0$ for $x\notin [b_0,\infty) \setminus\E$, and $C > 0$ is a normalization constant.
\end{lemma}
\begin{proof}
	For finite-gap sets, this is a reformulation of the Schwarz--Christoffel mapping. If $\E$ has infinitely many gaps, we consider them labelled by $j\in\bbN$ in an arbitrary way and denote $\E_n = [b_0,\infty) \setminus \cup_{j=1}^n (a_j,b_j)$. Denote by $M_n$ the Martin functions at $\infty$ corresponding to the sets $\E_n$, normalized by $M_n(z_*) = 1$ for some fixed $z_* < b_0$. Since the functions $M_n$ are all positive harmonic on $\C\setminus [b_0,\infty)$, for any $R > \lvert b_0\rvert$, by Harnack's principle they are uniformly bounded on the line segments parametrized by $-R+it$, $t+iR$, $t-iR$, with $t \in [-R,R]$. Since $M_n(x+iy)$ are increasing in $y > 0$ and symmetric, it follows that $M_n$ are uniformly bounded above on the boundary of $(-R,R)\times (-R,R)$ for any $R$ large enough. Since they are also nonnegative, they are a precompact sequence of subharmonic functions on $\C$. By the upper envelope theorem, for any subsequential limit $h = \lim_{k\to\infty} M_{n_k}$, quasi-everywhere on $\E$, $h(z) = \lim_{k\to\infty} M_{n_k}(z) = 0$, so by Theorem~\ref{lem:ConePositive}, $h$ is Martin function for the domain $\C\setminus\E$ with $h(z_*) =1$. It follows that $M_n$ converge to $h$ in $\cD'(\C)$.
	
	It follows that $\Theta_{n}$ converge to $\Theta$ since their real parts converge and their imaginary parts are zero on $(-\infty,b_0)$. In particular, the Herglotz functions $i\Theta'_{n}$ converge to $ci\Theta'$ uniformly on compact subsets of $\C_+$, so by interpreting this convergence in terms of their exponential Herglotz representations, \[
	\lim_{n\to\infty} \int_\R g(x) \xi_n(x) \frac{\dd x}{1+x^2} = \int_\R g(x) \xi(x) \frac{\dd x}{1+x^2}, \qquad \forall g \in C(\R\cup\{\infty\})
	\]
	where $\xi$ is determined by $\lim_{y\downarrow 0} \arg\Theta'(x+iy) = \pi \xi(x)$ Lebesgue-a.e.\ $x\in\R$. By using test functions $g$ supported in $(a_j,b_j)$, it follows that for each $j$, the critical points $c_{j,n}$ must converge to a point $c_j \in [a_j,b_j]$. Then $\xi_n$ converge pointwise to the function $\tilde\xi$ which is $1$ on intervals $(a_j,c_j)$, $-1$ on $(c_j,b_j)$, and $0$ on $[b_0,\infty)$, so by dominated convergence with dominating function $\lVert g \rVert_\infty \frac 1{1+x^2} \chi_{[b_0,\infty)\setminus\E}$,
	\[
	\lim_{n\to\infty} \int_\R g(x) \xi_n(x) \frac{\dd x}{1+x^2} = \int_\R g(x) \tilde\xi(x) \frac{\dd x}{1+x^2}, \qquad \forall g \in C(\R\cup\{\infty\}).
	\]
	Of course, this implies $\xi = \tilde \xi$ which implies \eqref{eqn:productformula}. Finally, by separating the contribution from the gap $(a_j,b_j)$ from the remainder of the integral, \eqref{eqn:productformula} can be extended into the gap $(a_j,b_j)$ to show that $i\Theta' > 0$ on $(a_j,c_j)$ and $i \Theta' < 0$ on $(c_j,b_j)$. It follows that $M' > 0$ on $(a_j,c_j)$ and $M' < 0$ on $(c_j,b_j)$, so our construction of $c_j$ as limits of $c_{j,n}$ satisfies the property in the lemma.
\end{proof}

As the final topic of this section, we describe a class of Akhiezer--Levin sets for which it can be seen by purely complex theoretic arguments that the Martin function has the two-term expansion \eqref{Martinexp2term}. While this is not as general as Theorem~\ref{thm:MainAkhiezerLevin}, within its scope of applicability, it provides a formula for $a_\E$ in terms of critical points of the Martin function.

\begin{lemma}\label{lemma:constantfromfinitegaplength}
	Let $\E \subset \bbR$ be of the form \eqref{eq:spectrumGaps}. If $\sum_{j=1}^N (b_j - a_j) < \infty$, then $\E$ is an Akhiezer--Levin set, the Martin function obeys the two-term expansion \eqref{Martinexp2term}, and
	\begin{equation}\label{eqn:formulaforconstant}
	a_\E = b_0 + \sum_{j=1}^N (a_j + b_j - 2 c_j).
	\end{equation}
\end{lemma}

\begin{proof}
	Finite gap length can be restated as $\int \chi_{[b_0,\infty)\setminus\E}(x) \dd x < \infty$ and it implies that the exponent in \eqref{eqn:productformula} can be split into two separately integrable integrands, of which one is $z$-independent, to give
	\[
	i\Theta_\E'(z) = \frac{C_\E}{\sqrt{b_0-z}} e^{\int_{[b_0,\infty) \setminus \E} \xi(x) \frac{1}{x-z} \dd x}.
	\]
	For any $\delta > 0$, using finite gap length and dominated convergence,
	\[
	\int_{[b_0,\infty)\setminus \E} \xi(x) \frac 1{x-z} \dd x = - \frac 1z \int_{[b_0,\infty)\setminus \E} \xi(x) \dd x + o(\lvert z \rvert^{-1}),
	\]
	as $z\to\infty$,  $\arg z\in [\delta,2\pi-\delta]$. Evaluating the integral $\int_{[b_0,\infty)\setminus \E} \xi(x) \dd x$ and substituting into $\Theta'(z)$,
	\[
	i\Theta'_\E(z) = C_\E \left( \frac{1}{\sqrt{-z}} + \frac 12 (b_0 + \sum_{j=1}^N (a_j + b_j - 2 c_j)) \frac 1{\sqrt{-z}^{3}} + o(\lvert z \rvert^{-3/2}) \right)
	\]
	and integrating along rays shows that, as $z\to \infty$ with $\arg z \in [\delta,2\pi-\delta]$,
	\[
	i\Theta_\E(z) = C_\E\left( - 2 \sqrt{-z} + (b_0 + \sum_{j=1}^N (a_j + b_j - 2 c_j)) \frac 1{\sqrt{-z}} + o(\lvert z \rvert^{-1/2}) \right).
	\]
	Taking imaginary parts gives a two-term expansion of $M_\E$, which matches \eqref{Martinexp2term} with $C_\E = \frac 12$. Reading off the second term gives \eqref{eqn:formulaforconstant}.
\end{proof}

\bibliographystyle{amsplain}

\providecommand{\MR}[1]{}
\providecommand{\bysame}{\leavevmode\hbox to3em{\hrulefill}\thinspace}
\providecommand{\MR}{\relax\ifhmode\unskip\space\fi MR }
\providecommand{\MRhref}[2]{%
	\href{http://www.ams.org/mathscinet-getitem?mr=#1}{#2}
}
\providecommand{\href}[2]{#2}

\end{document}